\documentclass[a4paper, 10pt, twoside, notitlepage, reqno]{amsart}

\usepackage{amsmath,amscd}
\usepackage{amssymb}
\usepackage{amsthm}
\usepackage{comment}
\usepackage{graphicx}
\usepackage{xcolor}

\usepackage{mathrsfs}
\usepackage[ocgcolorlinks, linkcolor=blue]{hyperref}
\definecolor{redish}{rgb}{0.95,0.20,0.20}

\usepackage[ocgcolorlinks,linkcolor=blue]{hyperref}

\theoremstyle{plain}
\newtheorem{thm}{Theorem}[section]
\newtheorem{prop}{Proposition}[section]

\newtheorem{lem}[prop]{Lemma}

\newtheorem{defi}[prop]{Definition}
\newtheorem{rmk}[prop]{Remark}

\numberwithin{equation}{section}
\newcommand {\R} {\mathbb{R}} 
 \newcommand {\N} {\mathbb{N}}
\newcommand {\C} {\mathbb{C}} 
\newcommand {\p} {\partial}

\newcommand{\eps}{\epsilon}
\newcommand{\vareps}{\varepsilon}

\newcommand{\ol}{\overline}

\newcommand{\abs}[1]{\lvert #1 \rvert}          % Formatting for the absolute value
\newcommand{\norm}[1]{\lVert #1 \rVert}         % Formatting for the norm
\newcommand{\ccdot}{\,\cdot\,}
\newcommand{\s}{\hspace{0.5pt}}

\title[]{An inverse problem for a semilinear elliptic equation on conformally transversally anisotropic manifolds}

\author[A. Feizmohammadi]{Ali Feizmohammadi}
\address{The Fields Institute for Research in Mathematical Sciences, Toronto, Canada}
\curraddr{222 College St, Toronto, ON M5T 3J1}
\email{afeizmoh@fields.utoronto.ca}

\author[T. Liimatainen]{Tony Liimatainen}
\address{Department of Mathematics and Statistics, University of Helsinki, Helsinki, Finland}
\curraddr{}
\email{tony.liimatainen@helsinki.fi}

\author[Y.-H. Lin]{Yi-Hsuan Lin}
\address{Department of Applied Mathematics, National Yang Ming Chiao Tung University, Hsinchu, Taiwan}
\curraddr{}
\email{yihsuanlin3@gmail.com}

\begin{document}
	
	\maketitle

	\begin{abstract}
	Given a conformally transversally anisotropic manifold $(M,g)$, we consider the semilinear elliptic equation
	$$(-\Delta_{g}+V)u+qu^2=0\quad \text{on $M$}.$$
	We show that an a priori unknown smooth function $q$ can be uniquely determined from the knowledge of the Dirichlet-to-Neumann map associated to the semilinear elliptic equation. This extends the previously known results of the works~\cite{FO19,LLLS2019inverse}. Our proof is based on analyzing higher order linearizations of the semilinear equation with non-vanishing boundary traces and also the study of interactions of two or more products of the so-called Gaussian quasimode solutions to the linearized equation.

		\medskip
		
		
		%\noindent{\bf Mathematics Subject Classification (2010)}: 
		
	\end{abstract}

	\tableofcontents

\section{Introduction}
	
Let $(M,g)$ be a smooth compact Riemannian manifold of dimension $n\geq 3$ with a smooth boundary. We assume that $(M,g)$ is {\em conformally transversally anisotropic} (CTA), that is to say, 
\begin{equation}
\label{M_1}
M \Subset I \times M_0, 
\end{equation}
and the metric $g$ has a smooth extension to $\R\times M_0$ so that
\begin{equation}
\label{M_2}
g= c(x_1,x')(\,dx_1\otimes\,dx_1+ g_0(x')),
\end{equation}
where $(M_0,g_0)$ is a compact $(n-1)$-dimensional Riemannian manifold with a smooth boundary $\p M_0$ \cite{DosSantosFerreira2009}. Let $q, V$ be real-valued smooth functions on $M$ and consider the semi-linear elliptic equation:

\begin{equation}\label{pf}
\begin{aligned}
\begin{cases}
(-\Delta_{g}+V) u+qu^2=0, 
&\text{ in } M
\\
u= f 
&\text{ on }\p M
\end{cases}
    \end{aligned}
\end{equation}

We make the standing assumption that $0$ is not a Dirichlet eigenvalue for the operator $-\Delta_g+V$. As shown in \cite[Proposition 2.1]{LLLS2019inverse}, equation \eqref{pf} is well-posed for sufficiently small Dirichlet data $f$. Precisely, given any $\alpha \in (0,1)$, there exists $C,\delta >0$ such that for all 
$$ f \in U_\delta=\left\{h \in C^{2,\alpha}(\p M)\,|\, \|f\|_{C^{2,\alpha}(\p M)} \leq \delta \right\},$$ 
the equation \eqref{pf} has a unique solution $u$ in the set 
\begin{equation}
\label{u_energy}
\left\{w \in C^{2,\alpha}(M)\,|\, \norm{w}_{C^{2,\alpha}(M)}\leq C\delta \right\}.
\end{equation}
Moreover,
$$ \norm{u}_{C^{2,\alpha}(M)} \leq C\norm{f}_{C^{2,\alpha}(\p M)}.$$
We define the associated \emph{Dirichlet-to-Neumann} map (DN map in short) for \eqref{pf} by
\[ 
\Lambda_{q} f = \left. \p_\nu u \right|_{\p M} \text{ for }  \, f \in U_\delta,
\]
where $u$ is the unique solution to \eqref{pf} that lies in the set \eqref{u_energy} and $\nu$ denotes the unit outward normal vector field on $\p M$.

In this paper, we consider the following inverse problem: Given an a priori fixed CTA manifold $(M,g)$ and a smooth zeroth order coefficient $V$, is it possible to recover an a priori unknown function $q$ given the knowledge of the map $\Lambda_{q}$? We show that this is indeed possible under the following minor technical assumption on the transversal manifold $(M_0,g_0)$: 
\begin{itemize}
	\item[(H1)]{\it Given any $p\in M_0$, there exists a non-tangential geodesic $\gamma$, which has no self-intersections and  passes through $p$.}
\end{itemize}
Precisely, we prove the following uniqueness result.

\begin{thm}\label{t1}
Let $(M,g)$ be a conformally transversally anisotropic manifold of the form \eqref{M_1}--\eqref{M_2} and suppose that (H1) is satisfied. Let $V\in C^{\infty}(M)$ and assume that zero is not a Dirichlet eigenvalue for $-\Delta_g+V$ on $M$. Let $q_1,q_2 \in C^{\infty}(M)$ and assume that for some $\delta>0$ sufficiently small and for any  $f\in U_\delta$
\[
\Lambda_{q_1}f = \Lambda_{q_2}f.%, \quad \text{ for any } f\in U_\delta,
\]
%for some $\delta>0$ sufficiently small. 
Then
\[
  q_1=q_2 \quad \text{in $M$}.
\]
\end{thm}
We will provide a discussion of the assumption (H1) as well as the main novelties of Theorem~\ref{t1} in Section~\ref{key_ideas}.

\subsection{Previous literature}
Inverse problems for non-linear partial differential equations is a topic with a vast literature. When the manifold is assumed to be Euclidean, the first result goes back to the work Isakov and Sylvester in \cite{isakov1994global} where the authors considered the equation
$$-\Delta u +F(x,u)=0,$$
on a Euclidean domain of dimension greater than or equal to three and studied the problem of recovering a class of non-linear functions $F(x, u)$ that satisfy a homogeneity property as well as certain monotonicity and growth conditions on its partial derivatives. The analogous problem in dimension two was first solved by Isakov and Nachman in \cite{victorN}. For further results in Euclidean geometries, we refer the reader to the works 
%\cite{isakov1993uniqueness,MR3168271} in the context of parabolic equations, to
\cite{sun2004inverse, sun2010inverse,LLLS2019inverse,LLLS2021b,MR4052205} in the context semilinear elliptic equations, to \cite{sun1996,sun1997inverse,MR1944036,MR4138229,LLS18, CF20,MR4300916,MR4227095,MR3964222,MR4197837} in the context of quasilinear elliptic equations and to \cite{lai2022inverse,lin2020monotonicity} for fractional semilinear elliptic equations. We also mention the early work \cite{isakov1993uniqueness} and the work \cite{MR3168271} on similar results on Euclidean geometries for parabolic equations.

Most of the results discussed above are based on the idea of higher order linearization of nonlinear equations. The idea of a first or a second order linearization was initiated by in \cite{isakov1993uniqueness,isakov1994global} and the idea of higher order linearizations was introduced and developed fully by Kurylev, Lassas and Uhlmann \cite{kurylev2018inverse} in the context of nonlinear hyperbolic equations over Lorentzian geometries. There, the authors showed that in geometric settings, it is possible to solve certain classes of inverse problems for nonlinear hyperbolic equations in a much broader geometric generality compared to analogous inverse problems stated for linear hyperbolic equations. We refer the reader to the works \cite{wang2016quadratic,lassas2017determination,lassas2018inverse,MR4299822,hintz2020inverse,MR4182329,MR4336252,lassas2021stability} for more examples of inverse problems for nonlinear hyperbolic equations solved in broad Lorentzian geometries. We also point out the simultaneous recovery results \cite{lin2021determining,lin2021simultaneous} in inverse problems for semilinear parabolic and hyperbolic equations in the Euclidean space.

Recently, the works \cite{FO19,LLLS2019inverse} introduced a similar higher order linearization approach in the context of semilinear elliptic equations on CTA manifolds. We also refer the reader to the more recent works \cite{KrUh20,MR4332042} on study of similar inverse problems for nonlinear elliptic equations stated on CTA manifolds. In \cite{FO19,LLLS2019inverse}, it was proved that for elliptic semilinear equations of the form
\begin{equation}\label{lit_nonlinear}
-\Delta_g u + F(x,u)=0\quad \text{on $M$},
\end{equation}
with non-linear functions $F(x,z)$ that depend analytically on $z$, the problem of recovering the differentials $\p^k_z F(x,0)$ with $k\geq 3$ is equivalent to the question of injectivity of products of four solutions to the linearized equation
$$ (-\Delta_g+V) u=0\quad \text{on $M$}.$$
This density property was subsequently proved in \cite{FO19,LLLS2019inverse} without imposing any geometric assumptions on the transversal manifold $(M_0,g_0)$, through studying products of four Gaussian quasimode solutions to the linear equation. The underlying theme discovered in the latter works is that one can solve inverse problems for nonlinear elliptic equations in CTA manifolds without imposing additional strong assumptions on the transversal manifold $(M_0,g_0)$. This is in sharp contrast to the study of inverse problems for linear elliptic equations on CTA manifolds \cite{ferreira2009limiting,ferreira2013calderon} where additional strong assumptions must be imposed on the transversal manifold such as simplicity or existence of a strictly convex function on $(M_0,g_0)$. 

In this paper, we have considered an extension of \cite{FO19,LLLS2019inverse} that allows non-linearities $F(x,u)$ in \eqref{lit_nonlinear} that have a quadratic term in $u$. As far as we know, the only previous result that is concerned with recovery of quadratic non-linear functions on CTA manifolds is \cite[Theorem 2]{FO19} in the context of three and four dimensional CTA manifolds under additional geometric assumptions on the transversal manifold.

\subsection{Outline of the key ideas}
\label{key_ideas}

One of the key themes in the recent works that study inverse problems for nonlinear equations of the form 
$$-\Delta_g u + q u^m=0,\quad \text{on $M$},$$
on CTA manifolds $(M,g)$ with any integer $m\geq 2$ is the reduction from the problem of recovering the unknown coefficient $q$ to the density problem of showing that the products of $m+1$ harmonic functions on $(M,g)$ forms a dense set in $L^\infty(M)$. This reduction is based on an $m$-fold linearization argument for the nonlinear equation.

When $m\geq 3$, the latter density problem involves the product of four harmonic functions. Following the arguments of \cite{LLLS2019inverse} and choosing harmonic functions based on Gaussian quasimode constructions near four intersecting geodesics on the transversal factor $(M_0,g_0)$, the density property can be proved. However, when $m=2$, one only obtains product of three Gaussian quasimodes and this is not a sufficiently reach set to conclude our desired density claim. 

In this paper, we introduce a method to solve the coefficient determination problem concerning the case $m=2$, by considering further linearizations of the equation up to fourth order, rather than just considering the second order linearization of the equation. In this sense we over-differentiate the nonlinearity. This will allow us to \emph{implicitly} obtain products of more harmonic functions. We remark that in analyzing the third and fourth order linearization of the equation \eqref{pf}, one important step is to try to understand the interactions of two Gaussian quasimode solutions to the linearized equation, namely we need to analyze an equation of the form
$$ -\Delta_g w = fu_1 u_2\quad \text{on $M$},$$
where $u_1$ and $u_2$ are two Gaussian quasimodes. To the best of our knowledge, a treatment of the latter equation does not exist in the literature. We show that the latter equation can be solved asymptotically with respect to the semi-classical parameter of the Gaussian quasimodes and in doing so we obtain precise closed form expressions for $w$ modulo a small correction term, see Section~\ref{Sec:Gauss_interac}. This will be partly based on a WKB approximation for $w$ as well as a new Carleman estimate on CTA manifolds with boundary terms. As the correction term has a non-vanishing trace on $\p M$, we also need to introduce a variant of the higher order linearization method with a family of Dirichlet data that also depend on additional powers of the involved small parameters (see Section~\ref{sec:higher_ord_lin_with_bndr_vals} and also Section~\ref{Sec:proof_of_t1}). 

Let us remark in closing that our assumption (H1) simplifies the presentation of the Gaussian quasimode solutions to the linearized equation \eqref{linear_eq}. It is well known that Gaussian quasimodes for equation \eqref{linear_eq} can also be constructed in the absence of (H1), see for example \cite{ferreira2013calderon}. However, in this paper we impose the mild assumption (H1) to better convey the key ideas discussed above.

\subsection{Organization of the paper}

The paper is organized as follows. In Section \ref{Section 2}, we reduce our the setup of our study to a case where the conformal factor $c$ in \eqref{M_2} of the CTA manifold is constant $1$. There we also review the higher order linearization method, and derive the linearized equations and associated integral identities we use. We review suitable Gaussian quasimodes for the first linearized equations in Section \ref{sec:CGOs}. In Section \ref{Section 4}, we find solution formulas for the special solutions of the second and third linearized equations. In Section \ref{Sec:proof_of_t1} we prove Theorem \ref{t1} by utilizing these solutions. Finally, we prove a boundary determination result, derive Carleman estimates and compute coefficients related to products of solutions in Appendix. %\ref{Section: Boundary determination}, \ref{Section: Boundary Carleman} and \ref{Appendix D_ikl}, respectively.

\section{Preliminaries}\label{Section 2}

% 
% We regord some facts which will use  in the proof of Theorem \ref{t1}.

\subsection{Reduction to the case $c=1$}
We show that for our purposes we can assume without any loss in generality that $c \equiv 1$. This is standard, see e.g. \cite{DosSantosFerreira2009} or \cite[Section 2.3]{FO19}. To see this, let us define $\hat{g}=(dx_1)^2+g$ so that $g=c\hat{g}$. Using the transformation law for changes of the Laplace-Beltrami operator under conformal rescalings of the metric, we write
\begin{align}\label{conformal} 
	c^{\frac{n+2}{4}}(-\Delta_{g}u+Vu+qu^2)=-\Delta_{\hat{g}}v+ \hat{V}v+\hat{q}v^2,
\end{align}
where $v=c^{\frac{n-2}{4}}u$, $\hat{V}=cV-(c^{\frac{n-2}{4}}\Delta_{g}\,c^{-\frac{n-2}{4}})$ and $\hat{q}=c^{-\frac{n-2}{2}}q$. This shows that there exists a one to one correspondence between solutions to \eqref{pf} with $f \in U_\delta$ and solutions to the following equation
\\
\begin{equation}\label{pf1}
\begin{aligned}
\begin{cases}
(-\Delta_{\hat{g}}+\hat{V})v+\hat{q}v^2=0, 
& \quad x \in M
\\
v= h, 
&\quad x \in \p M
\end{cases}
    \end{aligned}
\end{equation}
provided that $\|c^{-\frac{n-2}{4}}h\|_{C^{s}(\p M)} \leq \delta$. Hence, the DN map for \eqref{pf} determines the DN map for \eqref{pf1}. Thus the problem of unique recovery of $q$ from the DN map for \eqref{pf} is equivalent to that of determining $\hat{q}$ from the DN map for \eqref{pf1}. With this observation in mind, for the remainder of this paper and without loss of generality, we assume that 
$c \equiv 1$ so that $$g=dx_1\otimes dx_1+g_0.$$

\subsection{Higher order linearization method with boundary values}\label{sec:higher_ord_lin_with_bndr_vals}
In this section, we discuss the higher order linearization method of equation \eqref{pf}. Our method is slightly different from the, by now standard, one \cite{LLLS2019inverse, FO19}. The difference is that we include boundary terms, which are not linear in the used small parameters.

Let $\eps_i\in \R$  and $f_i, f_{ij}, f_{i jk }\in C^{2,\alpha}(\p M)$, for some $0<\alpha<1$ and for $i,j,k =1,\ldots, 4$, and $\eps=(\eps_1,\eps_2,\eps_3,\eps_4)$. In the most general case of this paper, we take boundary values $f$ to be of the form 
\begin{align}\label{f_epsilon}
	f_\eps:=\displaystyle \sum_{i=1}^4\eps_i f_i+\sum_{i,j =1}^4 \eps_i \eps_j f_{ij}+\sum_{i,j,k =1}^4\eps_i\eps_j\eps_k f_{ijk} \quad \text{ on }\quad \p M.
\end{align}
Observe that $f_\eps \in U_\delta$ for sufficiently parameters  $\epsilon_i$, where $U_\delta$ is defined by 
\[
U_\delta:=\left\{ f\in C^{2,\alpha}(\p M)| \, \norm{f}_{C^{2,\alpha}(\p M)}<\delta  \right\},
\]
for some sufficiently small number $\delta>0$. 
By using the implicit function theorem and the Schauder estimate for linear second order elliptic equations, one can show that the solution $u_f$ to the nonlinear equation \eqref{pf} depends smoothly (in the Frech\'et sense) on the parameters $\epsilon_1,\ldots, \eps_4$ (see  \cite[Section 2]{FO19,LLLS2019inverse} for detailed arguments).

The first linearization of the equation \eqref{pf} at the zero boundary value is
\begin{equation}\label{linear_eq}
	\begin{aligned}
		\begin{cases}
			(-\Delta_{g}+V)v^{(i)}=0 
			& \text{ in } M,
			\\
			v^{(i)}=f_i
			&\text{ on }\p M,
		\end{cases}
	\end{aligned}
\end{equation}
for $i=1,2,3,4$. Here 
\[
v^{(i)}:= \left. \p_{ \epsilon_i}\right|_{\eps=0} u_f,
\]
where  we have denoted $\eps=0$ for the case $\eps_1=\eps_2=\eps_3=\eps_4=0$.
The second linearization 
%\f{Yi-Hsuan, why are these normalizations by $1/2$? If we use the normalization, please check that the $1/2$ works also for third and fourth order linearizations. -Tony}
\[
w^{(ij)}:= \left. \p^2_{ \epsilon_i \epsilon_j}\right|_{\eps=0} u_f
\]
of $u_f$ satisfies the second linearized equation \eqref{linear_eq}
\begin{align}\label{w_eq}
	\begin{cases}
		(-\Delta_{g}+V)w^{(ij)}=-2qv^{(i)}v^{(j)}
		&\text{ in }M,
		\\
		w^{(ij)}= 
		f_{ij}
		&\text{ on } \p M,
	\end{cases}
\end{align}
for different $i,j=1,2,3, 4$, where the functions $v^{(i)}:=\left. \p_{\p \eps_i}\right|_{\eps =0}u_f$ are the unique solutions to the first linearized equation

We denote 
\begin{align*}
	w^{(ijk)}:=\left. \p^3_{ \eps_i  \eps_j  \eps _k}\right|_{\eps=0}u_f \quad  \text{ and }\quad w^{(1234)}:=\left.\p^4_{ \eps_1  \eps_2 \eps _3  \eps_4}\right|_{\eps=0}u_f,
\end{align*}
and see that they satisfy 
\begin{align}\label{w_eq 3rd}
	\begin{cases}
		(-\Delta_{g}+V)w^{(ijk)}=-2q\left( v^{(i)} w^{(jk)} + v^{(j)}w^{(ik)} + v^{(k)} w^{(ij)}\right) 
		&\text{ in }M,
		\\
		w^{(ijk )}= f_{ijk} 
		&\text{ on } \p M,
	\end{cases}
\end{align}
for different $i,j,k = 1,2,3,4$, and
\begin{align}\label{w_eq 4th}
	\begin{cases}
		(-\Delta_{g}+V)w^{(1234)}
		=-2q \left(  v^{(1)}w^{(234)} + v^{(2)} w^{(134)}   \right. \\		 
		\qquad \qquad \qquad \qquad  \qquad \qquad + v^{(3)} w^{(124)}+ v^{(4)}w^{(123)} \\
		\qquad \qquad \qquad \qquad  \qquad \qquad \left.  + w ^{(12)}w^{(34)}+   w^{(13)}w^{(24)} + w^{(14)} w^{(23)}  \right) 
		&\text{ in }M,
		\\
		w^{(1234)}= 0 
		&\text{ on } \p M.
	\end{cases}
\end{align}
We will construct special solutions for the above linearized equations in Section \ref{sec:CGOs}.

\subsection{Integral identities for the inverse problem}

Let us consider two potentials $q_1,q_2 \in C^{\infty}(M)$. Let $v^{(i)}$, $w_\beta^{(ij)}$, $w_\beta^{(ijk)}$ and $w_\beta^{(1234)}$ be the respective solutions of \eqref{linear_eq}, \eqref{w_eq}, \eqref{w_eq 3rd} and \eqref{w_eq 4th}, where the index $\beta=1,2$ refers to the potentials $q=q_\beta$, and $i,j,k=1,2,3,4$. We denote by $v^{(5)}$ an additional solution to the linearized equation:
	\begin{align*}
			\begin{cases}
				(-\Delta_{g}+V)v^{(5)}=0 
				& \text{ in } M,
				\\
				v^{(5)}=f_5
				&\text{ on }\p M.
			\end{cases}
		\end{align*}
We record the integral identities for the second, third and fourth order linearized equations. 

\begin{lem}[Integral identities]\label{Lem:Integral identities}
	Let $f_1,\ldots, f_5\in C^{2,\alpha}(\p M)$ and $i,j,k\in \{1,2,3,4\}$. The following integral identities hold:
	
	\noindent\textbf{(1)}
 The second order integral identity
		\begin{align}\label{second integral id}
			\begin{split}
				 \int_{\p M}   \left.\p^2_{\epsilon_i\epsilon_j} \right|_{\epsilon=0} \left( \Lambda_{q_1} f_\epsilon -\Lambda_{q_2} f_\epsilon  \right) f_k\, dS =2  \int_M (q_1-q_2) v^{(i)}v^{(j)}v^{(k)} \, dV.
			\end{split}
		\end{align}
		
\noindent\textbf{(2)} The third order integral identity
		\begin{align}\label{third integral id}
			\begin{split}
				\int_{\p M}&\left. \p^3_{\eps_i \eps_j \eps_k}\right|_{\eps=0}\left( \Lambda_{q_1}f_\eps-\Lambda_{q_2}f_\eps \right)  f_m\, dS  \\
				=& 2\int_{M} \left\{ q_1\left( v^{(i)} w_1^{(jk)} + v^{(j)}w_1^{(ik)} + v^{(k)} w_1^{(ij)}\right)  \right. \\
				& \quad  \qquad \left. - q_2\left( v^{(i)} w_2^{(jk)} + v^{(j)}w_2^{(ik)} + v^{(k)} w_2^{(ij)}\right)  \right\} v^{(l)}\, dV.
			\end{split}
		\end{align}
		%for different $i, j ,k=1,2,3,4$.
		
\noindent\textbf{(3)} The fourth order integral identity
		\begin{align}\label{fourth integral id}
			\begin{split}
				\int_{\p M}& \left.\p^4 _{\eps_1 \eps_2 \eps_3\eps_4}\right|_{\epsilon=0} \left(   \Lambda_{q_1} f_\epsilon- \Lambda_{q_2} f_\epsilon \right) f_5\, dS \\
				= & 2 \int_M \Bigg\{ q_1 \left(  v^{(1)}w_1^{(234)} + v^{(2)} w_1^{(134)}  + v^{(3)} w_1^{(124)}+ v^{(4)}w_1^{(123)}  \right.  \\
				&\qquad\qquad \qquad \qquad \qquad  \qquad  \left.  + w_1^{(12)}w_1^{(34)}+   w_1^{(13)}w_1^{(24)} + w_1^{(14)} w_1^{(23)}  \right) \\
				& \ \quad  \qquad - q_2 \left(  v^{(1)}w_2^{(234)} + v^{(2)} w_2^{(134)}    + v^{(3)} w_2^{(124)}+ v^{(4)}w_2^{(123)}\right. \\
				& \left.\qquad \qquad \qquad \qquad \qquad \qquad    + w _2^{(12)}w_2^{(34)}+   w_2^{(13)}w_2^{(24)} + w_2^{(14)} w_2^{(23)}  \right)  \Bigg\} v^{(5)}\, dV,
			\end{split}
		\end{align}

\end{lem}
\begin{proof}
	The proof is based on integration by parts. We only prove \eqref{second integral id} explicitly. The other two integral identities follow similarly. 
	
	\medskip
	
\noindent\textbf{(1)} Let us consider the second linearized equation \eqref{w_eq} with $q=q_\beta$ for $\beta=1,2$. Integrating by parts yields
		\begin{align*}
			&\int_{\p M} \left.\p^2_{\eps_i \eps_j }\right|_{\eps=0}\left( \Lambda_{q_1}f_\eps-\Lambda_{q_2}f_\eps \right)  f_k \, dS \\
			= &\int_{\p M} \left( \p_\nu w_1^{(ij)}-\p_\nu w_2^{(ij)} \right) f_k \, dS \\
			=& \int_{M} \left(\Delta w_1^{(ij)} -\Delta w_2^{(ij)} \right) v^{(k)} \, dV + \int_{M} \nabla \left( w_1^{(ij)} - w_2^{(ij)} \right) \cdot \nabla  v^{(k)} \, dV  \\
			=& \int_{M} V\left(  w_1^{(ij)} - w_2^{(ij)} \right) v^{(k)} \, dV  +2\int_M (q_1-q_2)v^{(i)}v^{(j)}v^{(k)}\, dV\\
			&\qquad   - \int_{M} \left( w_1^{(ij)}-w_2^{(ij)}\right) \Delta v^{(k)}\, dV \\
			=&2\int_M (q_1-q_2)v^{(i)}v^{(j)}v^{(k)}\, dV,
		\end{align*}
		where we have utilized $w_1^{(ij)}=f_{ij}=w_2^{(ij)}$ on $\p M$ and $(-\Delta_g+V) v^{(k)}=0$ in $M$. %, for $k=1,2,3,4$.
		
		\medskip	
		
\noindent\textbf{(2)} We have 	
		\begin{align*}
			\int_{\p M} \left. \p^3_{\eps_i \eps_j \eps_k}\right|_{\eps=0}\left( \Lambda_{q_1}f_\eps-\Lambda_{q_2}f_\eps \right)  f_l\, dS 
			= \int_{\p M} \left( \p_\nu w_1^{(ijk)}-\p_\nu w_2^{(ijk)} \right) f_l\, dS.
		\end{align*}
		The above integration by parts combined with the equations \eqref{linear_eq} and \eqref{w_eq 3rd} results in the claimed identity. Proof of \textbf{(3)} is obtained similarly.
% 		
% 			\medskip
% \noindent\textbf{(3)} It is similar to previous cases, so we skip the proof.
	\end{proof}

\section{Complex geometrical optics and Gaussian beam quasimodes}\label{sec:CGOs}

Let us introduce the \emph{complex geometric optics} type solutions for the first order linearized equation. These are solutions to the linearized equation \eqref{linear_eq} that concentrate on planes of the form $I\times \gamma$, where $I$ is an interval and $\gamma$ is an inextendible non-tangential geodesic on $M_0$. We call them CGOs in short. We also assume in this paper for simplicity that $\gamma$ does not have self-intersections.

We recall the Gaussian quasimode construction for the equation \eqref{linear_eq} that originated from \cite[Section 3]{ferreira2013calderon} in the setting of CTA manifolds.  We follow the constructions \cite[Section 4.1, Proposition 5.1]{FO19} and \cite[Section 5 and Appendix]{LLLS2019inverse} that allow a zeroth order term $V$ in \eqref{linear_eq} as well as providing decay estimates in higher order Sobolev spaces. We refer to these works for details of the constructions in this section.

 Throughout the remainder of this paper and for the sake of brevity of notation, we write $\Delta\equiv \Delta_g$ and $\nabla\equiv \nabla _g$. We also denote by $\nabla'$ the gradient operator in the transversal variables on $M_0$. We first consider a unit speed non-tangential geodesic $\gamma:[l_1,l_2]\to M_0$ that connects two points on the boundary $\p M_0$. We assume that $\gamma$ does not have self-intersections for simplicity. We write $(\hat{M}_0,g_0)$ for an artificial smooth extension of $(M_0,g_0)$ into a slightly larger smooth Riemannian manifold and denote by $(t,y)$ the Fermi coordinates in a tubular neighborhood of the geodesic $\gamma$, where $t\in [l_1,l_2]$ and $y\in B_{\delta'}(\R^{n-2})$ for some $\delta'>0$ sufficiently small. We refer the reader to \cite[Section 3]{ferreira2013calderon} for the details of the construction of Fermi coordinates. 

We define the complex parameter
$$ s= \tau+\mathbf{i}\lambda, \quad \tau>0,\quad\lambda \in \R,$$
where $\mathbf{i}=\sqrt{-1}$, $\lambda$ is to be viewed as a fixed parameter, and $\tau>0$ is an asymptotic parameter that tends to infinity. Given any $K>0$ and $N>0$, there exists a positive integer $N'$ depending on $K, N$ (see \eqref{N'_choice} for the precise choice) and solutions $v_s$ to the linear equation
$$(-\Delta + V)v_s=0 \quad \text{in $M$},$$
of the form
\begin{align}\label{form_1}
	v_{s}(x_1,t,y)&= e^{\pm s x_1} \left(\tau^{\frac{n-2}{8}}e^{\mathbf{i}s \psi(t,y)} a_s(x_1,t,y) + r_s(x_1,t,y) \right),
\end{align}
where  $\chi$ is a cutoff function supported in a $\delta'$-neighborhood of the origin and each term in the right hand side has certain properties that we will describe next. Alternatively, we denote $v_s$, $a_s$ and $r_s$ by $v_\tau$, $a_\tau$ and $r_\tau$ respectively.

The phase function $\psi(t,y)$ satisfies 
\begin{align} \label{phase_pr} 
\psi(\gamma(t))=t,\quad \nabla\psi(\gamma(t))= \dot{\gamma}(t),\quad \mathrm{Im}( D^2\psi(\gamma(t))) \geq 0,\quad \mathrm{Im}( D^2\psi|_{\dot{\gamma}(t)^\perp}) >0.
\end{align}
More explicitly, in terms of the Fermi coordinates we can write
$$\psi(t,y)= t+ \frac{1}{2}\sum_{j,k=1}^{n-2} H_{jk}(t)y_jy_t+O(|y|^3),$$
where the complex-valued symmetric matrix $H(t)= (H_{jk}(t))_{j,k=1}^{n-2}$ is given by the expression
$$H(t)= \dot{Y}(t)Y^{-1}(t),\quad \text{ for any }\,  t\in [l_1,l_2],$$
and $Y$ is a non-degenerate matrix that solves the second order linear differential equation
$$ \ddot{Y} + D Y=0\quad \text{ for any } \, t\in [l_1,l_2].$$
Here, the symmetric matrix $D$ is given by $D_{jk}=\frac{1}{2}\p^2_{jk}g^{11}$ for each $j,k=1,\ldots,n-2$. The matrix $H$ additionally satisfies 
$$ \textrm{Im}(H)(t)>0\quad\text{ for any } \, t\in [l_1,l_2],$$
and
$$ \det (\textrm{Im}(H(t)))\cdot |\det Y(t)|^2 = 1.$$

Next we describe the amplitude function in the expansion \eqref{form_1}. The amplitude $a_s(x_1,t,y)$ is of the form 
\begin{equation}\label{amplitude}
a_s(x_1,t,y) = \left( a_{0}(t,y) +\frac{a_1^{\pm}(x_1,t,y)}{s}+\cdots+\frac{a^{\pm}_{N'-1}(x_1,t,y)}{s^{N'-1}}\right)\chi\left(\frac{|y|}{\delta'}\right),
\end{equation}
where the principal amplitude $a_{0}(t,y)$ itself is given by the expression
$$ a_{0}(t,y)=a_{0,0}(t) +a_{0,1}(t,y)+\cdots+a_{0,N'-1}(t,y).$$
Here, $a_{0,0}(t)$ is an explicit positive function on $\gamma$ given by the expression
\begin{equation}
\label{ampl_0_pr}
a_{0,0}(t) = (\det Y(t))^{-\frac{1}{2}},
\end{equation}
and the subsequent terms $a_{0,j}(t,y)$ with $j=1,2,\ldots,N'-1$ are  homogeneous polynomials of degree $j$ in the $y$-coordinates. These terms arise as solutions to certain transport equations along the geodesic $\gamma$ on $M_0$.

The remaining amplitude terms $a_{k}^{\pm}$, for $k=1,2,\ldots,N'-1$, have analogous expressions of the form
$$ a_{k}^{\pm}(x_1,t,y)=a^{\pm}_{k,0}(x_1,t,y) +a_{k,1}^{\pm}(x_1,t,y)+\cdots+a^{\pm}_{k,N'-1}(t,y),$$
where $a_{k,j}^{\pm}$ are homogeneous polynomials of degree $j$ in the $y$-coordinates, for $j=0,1,\ldots, N'-1$. These amplitudes arise as solutions to certain complex transport equations on the plane $y=0$ on $M$ (see \cite[Section 4]{FO19} for more details). 

Finally, using \cite[Proposition 2, Lemma 4]{FO19} and fixing the order
\begin{equation}
\label{N'_choice}
N'=2+2N+2K,
\end{equation}
for the Gaussian quasimode construction, it follows that the remainder term in \eqref{form_1} satisfies the decay estimate
\begin{equation}
\label{gauss_bounds}
\norm{r_{s}}_{H^K(M)} \lesssim \tau^{-N}.
\end{equation}

\section{Solutions for the linearized equations}\label{Section 4}
We discussed CGO solutions for the first order linearization of the equation $(-\Delta  +V)u+qu^2=0$ at the zero solution \label{Sec:Gauss_interac} in the previous section. In this section, we construct solutions for the second and third order linearizations of the equation.

\subsection{Solutions for the second order linearization}
In what follows, we assume that geodesics do not have self-intersections.  Let $p_0 \in M_0$ and let $\gamma_1$ be a non-tangential geodesic passing through $p_0$ in some direction $v \in S_{p_0}M_0$. Here $S_{p_0}M_0$ stands for unit length vectors of $T_{p_0}M_0$. We will use the following definition.
\begin{defi}
   We say that two geodesics intersect \emph{properly} if they intersect and are not reparametrizations of each other.
\end{defi} 

Assume that $\gamma_2$ is another nontangential geodesic, and that it intersects $\gamma_1$ properly at $p_0$. If $v' \in S_pM_0$ is the velocity vector of $\gamma_2$ at $p_0$, then $v'$ is linearly independent of $v$ due to the uniqueness of geodesics. 
% passing through $p$ with direction $v'$.  Then, if $v'$ is sufficiently close to $v$, or to $-v$, the geodesic $\gamma_2$ will also be a non-tangential geodesic, which does not have self-intersections. 
Due to a compactness argument (see e.g. \cite{LLLS2019inverse}), the geodesics $\gamma_1$ and $\gamma_2$ can only intersect at a finite number of points. 
%We will use the notation $p_j$ with $j=1,\ldots,d$ to denote the intersection points. Note that $p_0$ is one of these intersection points.

We consider CGO solutions $v^{(1)}=v^{(1)}_\tau$ and  $v^{(2)}=v^{(2)}_\tau$ to the equation \eqref{linear_eq} corresponding to geodesics $\gamma_1$ and $\gamma_2$, respectively. That is, the CGOs $v^{(k)}$ for $k=1,2$ are of the form \eqref{form_1}:
\begin{equation}\label{cgo_sol_1}
\begin{aligned}
v^{(k)} &= e^{\pm s_k x_1}\left( \tau^{\frac{n-2}{8}} e^{\mathbf{i}s_k \psi_k} a^{(k)}_\tau + r_\tau^{(k)}\right),
\end{aligned}
\end{equation}
where $\psi_k$, $a_\tau^{(k)}$ and $r_\tau^{(k)}$ have the properties described in the previous section. %, for $k=1,2$.
%with the notational exception that the amplitudes $a_\tau^{(k)}$ are constructed up to order $N'-1$ instead of $N'$.  
We have also denoted 
\begin{align}\label{notation convention}
	s_k=c_k\tau+\mathbf{i}\lambda_k,
\end{align}
where $c_k,\lambda_k\in \R$, and $\tau>0$ is a (large) parameter. 

 In the next lemma, we construct solutions for the second linearized equation. After proving the lemma, in Proposition \ref{2nd_lin_sol_fixed_bndr}, we show that if the DN map is known, the boundary value of the solutions can be fixed.

\begin{lem}\label{2nd_lin_sol_free_bndr}
 Let $K,N\in \N\cup \{0\}$. Assume that $v^{(1)}$ and $v^{(2)}$ are CGOs, which correspond to properly intersecting geodesics on $M_0$ and are of the form \eqref{cgo_sol_1}. If the restrictions of the amplitudes $a^{(k)}$ to $M_0$ are supported in small enough neighborhoods of the geodesics $\gamma_k$ for $k=1,2$, and $N'=N'(K,N)$ is large enough, then the equation 
\begin{align}\label{eq:2nd_lin_sol_free_bndr}
(-\Delta_{g}+V)w=-2qv^{(1)}v^{(2)}
&\text{ in }M
    \end{align}
has a smooth solution $w$ up to the boundary $\p M$ with the following properties: The solution $w$ is of the form
\[
 w=w_0+e^{\tau \Psi}R,
\]
where
\[
	w_0=\tau^{\frac{n-2}{4}}e^{(\pm s_1\pm s_2) x_1+\mathbf{i}(s_1 \psi_1+s_1 \psi_1)}b_\tau,
	%w_0=\tau^{\frac{n-2}{4}}e^{\tau\Psi}e^\Lambda b_\tau,
	\]
    with
	\begin{align}\label{expansion_formula_new}
    \begin{split}
		b_\tau&=\frac{1}{\tau^2}b_{-2}+\frac{1}{\tau^3}b_{-3} +\cdots + \frac{1}{\tau^{2N'}}b_{-2N'},
		 \\
	b_{-2}&=\frac{2q}{(\pm c_1\pm c_2)^2-\abs{c_1\nabla'\psi_1+c_2\nabla'\psi_2}^2}a_0^{(1)}a_0^{(2)}.
	\end{split}
	\end{align}
	The function $\Psi$ is given by
	\begin{equation}\label{eq:psi_def}
	 \Psi=(\pm c_1\pm c_2)x_1+\mathbf{i}c_1 \psi_1+\mathbf{i}c_2 \psi_2.
	\end{equation}
    and $R=R_\tau$ is a remainder term that satisfies
	\[
\norm{R_\tau}_{H^{K}(M)}\lesssim \tau^{-N}.
	\]
\end{lem}
\begin{proof}
We first find an approximate solution for the equation 
	\begin{align}\label{second linearized in WKB_2}
		%\begin{cases}
			(-\Delta +V) \hat w_0= -2qV^{(1)}_\tau V^{(2)}_\tau&\text{ in } M,
			%w^{(12)}_0=0  &\text{ on }\p M.
		%\end{cases}
	\end{align}
	where
	     \begin{align*}
     	V^{(k)}_\tau=&e^{\pm s_k x_1}e^{\mathbf{i}s_k \psi_k} a^{(k)}_\tau. 
     \end{align*}
    %and $a_\tau^{(k)}$ are as in \eqref{amplitude}. 
    %so that
    %\[
    % a_\tau^{(k)}=\tau^{\frac{n-2}{4}} \left( a_0^{(k)} + \mathcal{O}(\tau^{-1}) \right).
    %\]
	%The 
	After that, we scale $\hat w_0$ and correct it by using either Carleman or elliptic estimates to a solution of \eqref{eq:2nd_lin_sol_free_bndr}. Here $\psi_k$ and $a^{(k)}$ are constructed with respect to geodesics $\gamma_k$ that intersect properly on the transversal manifold $M_0$.
	
	We shorthand our notation and write
	\[
	 e^{\pm s_1 x_1}e^{\mathbf{i}s_1 \psi_1}e^{\pm s_2 x_1}e^{\mathbf{i}s_2 \psi_2}:=e^{\tau \Psi}e^{\Lambda},
	\]
	where $\Psi$ is as in \eqref{eq:psi_def} and 
% 	\[
% 	 \Psi=(\pm c_1\pm c_2)x_1+\mathbf{i}c_1 \psi_1+\mathbf{i}c_2 \psi_2,
% 	\]
% 	as given in , and
	\[
	 \Lambda=\mathbf{i}(\pm \lambda_1\pm \lambda_2)x_1-\lambda_1 \psi_1-\lambda_2 \psi_2.
	\]
Using the expressions for the amplitude functions \eqref{amplitude}, the equation \eqref{second linearized in WKB_2} can be written as
	\[ 
	(\Delta_g-V)\hat w_0= e^{\tau \Psi} e^\Lambda\sum_{k=0}^{2(N'-1)} \frac{E_{-k}}{\tau^k},
	 \]
	where the functions $E_{-k}\in C^\infty(M)$, $k=0,1,\ldots, 2(N'-1)$, are supported near the intersection points of the geodesics $\gamma_1$ and $\gamma_2$. We have %Especially
	\[
	E_0=2qa_0^{(1)}a_0^{(2)}.
	\]
	
	Let us consider a \emph{WKB ansatz} for $\hat w_0$ of the form
\[
e^{\tau\Psi}\hat b. %+ \hat{R}_\tau,
\]
A direct calculation shows that 
\begin{align*}
		&(\Delta-V)\left(e^{\tau\Psi}\hat b\right)=e^{\tau\Psi}\left(\tau^2\langle \nabla \Psi, \nabla \Psi\rangle \hat b+\tau[2\langle \nabla \hat b,\nabla \Psi\rangle +\hat b(\Delta \Psi)]+(\Delta-V) \hat b\right), \nonumber
\end{align*}
where $\langle \eta , \zeta \rangle$ denotes the complexified Riemannian inner product. At the center of normal coordinates $\langle \eta , \zeta \rangle=\eta \cdot \zeta=\displaystyle\sum_{i=1}^n \eta_i \zeta_i$, for any $\eta, \zeta\in \C^n$. Note that $\langle\ccdot,\ccdot\rangle$ is not a Hermitean inner product of complex vectors. Especially $\langle\eta,\eta\rangle=0$ does not imply $\eta=0$.
We assume that $\hat b_\tau$ is an amplitude function of the form 
\begin{align}\label{form of b in s}
	\hat b_\tau=\frac{1}{\tau^2}\hat b_{-2}+\frac{1}{\tau^3}\hat b_{-3}+\cdots+\frac{1}{\tau^N}\hat b_{-2N'}. 
\end{align}

%The functions $E_{-k}$ are supported on neighborhoods of the intersection points of the geodesics $\gamma_1$ and $\gamma_2$. 
At an intersection point of the geodesics, we have by the properties of Gaussian beams (see \eqref{phase_pr}) that
\[
\nabla\Psi=(\pm c_1\pm c_2)e_1 + \mathbf{i}c_1\nabla' \psi_1+\mathbf{i}c_2 \nabla'\psi_2=(\pm c_1\pm c_2)e_1 + \mathbf{i}(c_1 \dot{\gamma}_1+c_2 \dot{\gamma}_2),
\]
where $\dot{\gamma}_1$ and $\dot{\gamma}_2$ are the velocity vectors of $\gamma_1$ and $\gamma_2$ at the intersection point. Here $e_1=\p_{x_1}x$, for $x=(x_1,\ldots, x_n)$. Since the geodesics $\gamma_1$ and $\gamma_2$ intersect properly 
\begin{align*}
 \langle \nabla \Psi, \nabla \Psi\rangle&=(\pm c_1\pm c_2)^2-\abs{c_1\nabla' \psi_1+c_2\nabla' \psi_2}^2 \\
 &=c_1^2\pm 2c_1c_2+  c_2^2-c_1^2-2c_1c_2\langle \dot{\gamma}_1, \dot{\gamma}_2\rangle-c_2^2 \\
 &=-2c_1c_2 \left( \langle \dot{\gamma}_1, \dot{\gamma}_2\rangle\mp 1 \right)\neq 0
\end{align*}
at the intersection points of the geodesics. 
By the above and assuming that $a^{(1)}$ and $a^{(2)}$ are supported in small enough neighborhoods of $\gamma_1$ and $\gamma_2$ we have
\begin{equation}\label{eq:division_by_Psi_ok}
|\langle \nabla \Psi, \nabla \Psi\rangle| \geq \text{ constant } > 0 
\end{equation}
on the support of each $E_{-k}$ for all $k=0,1,\ldots,2(N'-1)$.
% Consequently, we have:
% \[
% \quad (\Delta-V)\left(e^{\mathbf{i}(\lambda_1+\lambda_2) x_1}\tau^{\frac{n-2}{4}}e^{\tau\Psi^{(12)}}b_\tau\right)-qv_\tau^{(1)}v_\tau^{(2)}=e^{\mathbf{i}(\lambda_1+\lambda_2) x_1}\tau^{\frac{n-2}{4}}e^{i\tau \Psi}\tau^{1-N'} F_\tau,
% \]
% where $F_\tau:=\displaystyle\sum_{k=0}^{N'+1} \frac{F_k}{\tau^k}$ and $F_k$ are smooth functions that are independent of $\tau$.

% 	With these notations at hand, the equation \eqref{second linearized in WKB} reads
% 	\begin{align}
% 		\begin{cases}
% 			(-\Delta+V) \left(e^{\mathbf{i}(\lambda_1+\lambda_2) x_1}\tau^{\frac{n-2}{4}} e^{\tau \Psi^{(12)}} b_\tau \right) =qV^{(1)}_\tau V^{(2)}_\tau &\text{ in } M,\\
% 			\tau^{\frac{n-2}{4}}e^{\tau \Psi^{(12)}} b_\tau=0  &\text{ on }\p M.
% 		\end{cases}
% 	\end{align}
% 	and we have 
% 	\begin{align*}
% 	(-\Delta+V) \left(e^{\mathbf{i}(\lambda_1+\lambda_2) x_1}e^{\tau \Psi^{(12)}} b_\tau \right)  -qV^{(1)}_\tau V^{(2)}_\tau=e^{\mathbf{i}(\lambda_1+\lambda_2) x_1}	\tau^{\frac{n-2}{4}}e^{i\tau \Psi}\tau^{1-N'} F_\tau,
% 	\end{align*}
%     where $F_\tau:=\displaystyle\sum_{k=0}^{N'+1} \frac{F_k}{\tau^k}$ and $F_k$ are smooth functions that are independent of $\tau$.
	Let us set  
	$
	\hat b_0=\hat b_{-1}=0
	$ 
	and define the coefficients $\hat b_{-k}$ for $k=2,\ldots, 2N'$ recursively by the formula 
	\begin{align}\label{transport_equation_new}
	 \begin{split}
	 		\hat b_{-k}= \frac{e^\Lambda E_{-k+2}-[2\langle\nabla \hat b_{-k+1}, \nabla \Psi\rangle+\hat b_{-k+1}\Delta\Psi]-(\Delta-V)\hat b_{-k+2}}{\langle \nabla \Psi, \nabla \Psi\rangle}.
	 \end{split}
	\end{align}
Especially, 
\begin{align*}%\label{b_2}
\hat b_{-2}=e^{\Lambda}\frac{2q}{\langle \nabla \Psi, \nabla \Psi\rangle }  \,a_0^{(1)} a_0^{(2)}.
\end{align*}
We also see by a recursive inspection that $\hat b_k$ is supported on the set where \eqref{eq:division_by_Psi_ok} holds. Thus $\hat b_k$ are well-defined.
It follows by reindexing sums and using $\hat b_{-1}=\hat b_0=0$, such that 
\begin{align}\label{eq:telescopic_sum}
   	 &(\Delta-V)\left(e^{\tau\Psi}\hat b_\tau\right)- 2qV^{(1)}V^{(2)} \nonumber\\ 
   	 =&e^{\tau\Psi}\sum_{k=2}^{2N'}\left[\tau^{2-k}\langle \nabla \Psi, \nabla \Psi\rangle \hat b_{-k}+\tau^{1-k}[2\langle \nabla \hat b_{-k},\nabla \Psi\rangle +\hat b_{-k}(\Delta \Psi)] \right. \nonumber &\\
   	 & \qquad \qquad \left.+\tau^{-k}(\Delta-V) \hat b_{-k}- e^\Lambda E_{-k+2}\right] \nonumber\\
   	 =&e^{\tau\Psi}\sum_{k=2}^{2N'}\frac{1}{\tau^{k-2}}\left[\langle \nabla \Psi, \nabla \Psi\rangle \hat b_{-k}+[2\langle \nabla \hat b_{-k+1},\nabla \Psi\rangle + \hat b_{-k+1}(\Delta \Psi)] \right.\nonumber \\
   	 &\qquad \qquad \left.   +(\Delta-V)\hat b_{-k+2}- e^\Lambda E_{-k+2}\right]\nonumber\\
   	 &+e^{\tau\Psi}\Big(\tau^{-2N'+1} \left([2\langle \nabla \hat b_{-2N'},\nabla \Psi\rangle + \hat b_{-2N'}(\Delta \Psi)]+(\Delta-V) \hat b_{-2N'+1}\right)\nonumber \\
   	 &\qquad\qquad\qquad\qquad\qquad\qquad\qquad+\tau^{-2N'}(\Delta-V) \hat b_{-2N'}\Big)\nonumber\\
   	 =&e^{\tau\Psi}\mathcal{O}_{H^{\ell}M)}\left(\tau^{-2N'+1}\right),
\end{align}
since all the terms of the first sum after the second to last equality are zero by \eqref{transport_equation_new}. Here $\ell$ can be taken to be any number $0,1,2,\ldots$, since $\Psi$ and $\hat b_{-k}$ are smooth. 

Next we scale and correct $e^{\tau\Psi}\hat b_\tau$ so that it solves \eqref{eq:2nd_lin_sol_free_bndr}. We write 
\[
w= \tau^{\frac{n-2}{4}}e^{\tau\Psi}\hat b_\tau + \hat{R}_\tau.
\] 
Note that 
\[
 qv^{(1)}v^{(2)}=q\tau^{\frac{n-2}{4}}V^{(1)}V^{(2)}+qe^{\tau\Psi}r,
\]
where $r$ corresponds to the correction terms $r^{(1)}$ and $r^{(2)}$ given by 
\begin{align*}
 r&=r_\tau^{(1)}\tau^{\frac{n-2}{4}}e^{\mathbf{i}s_2\psi_2}a_{\tau}^{(2)}+r_\tau^{(2)}\tau^{\frac{n-2}{4}}e^{\mathbf{i}s_1\psi_1}a_{\tau}^{(1)}+r_\tau^{(1)}r_\tau^{(2)}.
 %&=e^\Psi r,
\end{align*}
Hence, $\hat{R}_\tau$ solves
\begin{equation}\label{eq:equation_for_hatR}
 (\Delta-V)\hat{R}_\tau=-\left[(\Delta -V)\tau^{\frac{n-2}{4}}e^{\tau\Psi}\hat b_\tau- 2\tau^{\frac{n-2}{4}}qV^{(1)}_\tau V^{(2)}_\tau\right]+2qe^{\tau\Psi}r,
\end{equation}
whenever $w$ solves \eqref{eq:2nd_lin_sol_free_bndr}.
Now, if $N'$ is chosen large enough, i.e.
$$N' \geq 2+2N+4K,$$
then we have $r=\mathcal{O}_{H^K(M)}(\tau^{-N})$ by combining the bounds \eqref{gauss_bounds} for the correction terms $r_\tau^{(\beta)}$, $\beta =1,2$ together with the bounds
\begin{equation}\label{eq:latter_estimate}
\tau^{\frac{n-2}{4}}\|e^{is_2\psi_2}\|_{H^K(M)}+\tau^{\frac{n-2}{4}}\|e^{is_2\psi_2}\|_{H^K(M)}\lesssim \tau^{K}.
\end{equation}
For example, see \cite[Lemma 4]{FO19} for the  estimate \eqref{eq:latter_estimate}.

Redefining $N'$ larger, if necessary, the equation \eqref{eq:equation_for_hatR} for $\hat{R}$ together with \eqref{eq:telescopic_sum} implies
\[
 (\Delta-V)\hat{R}_\tau=e^{\tau\Psi}\mathcal{O}_{H^K(M)}(\tau^{-N}).
\]
By writing 
\[
R_\tau =e^{-\tau\Psi}\hat{R}_\tau \quad \text{ and } \quad b_\tau=e^{-\Lambda} \hat b_\tau, 
\]
the claim follows from Carleman estimates \cite{DosSantosFerreira2009} if
\[
 \pm c_1\pm c_2\neq 0.
\]
Alternatively, if  
\[
 \pm c_1\pm c_2=0,
\]
 we may impose zero boundary conditions for $ R$ and use standard elliptic estimates. This completes the proof.
% Redefining $b_{-k}$ as $e^{-\Lambda}b_{-k}$ proves the . For the latter case we use zero boundary value for $\hat{R}$.
\end{proof}
\begin{rmk}\label{rem:second_lin_lemma}
% The function $\Psi$ is in general only defined on neighborhoods of the intersection points of the geodesics $\gamma_1$ and $\gamma_2$. As the support of $b_\tau$ is contained in these neighborhoods, the approximate solution $w_0$ of Lemma \ref{2nd_lin_sol_free_bndr} is nevertheless well defined on the whole $M$. 
% %Recall from \eqref{amplitude} that  is also defined on neighborhoods of the intersection points. 
%
We remark that in the case $\pm c_1\pm c_2\neq 0$, the correction term $R$ is a smooth function defined on an open manifold $U$ such that $M\Subset  U$, which satisfies 
$$
\norm{R}_{H^{K}(U)}\lesssim \tau^{-N},
$$ 
see \cite{DosSantosFerreira2009}.  In the case $\pm c_1\pm c_2=0$, the correction term $R$ has zero boundary values.
\end{rmk}

In addition, let us consider the second linearized equation \eqref{w_eq} for two possibly different potentials $q_1$ and $q_2$. We show that if $\Lambda_{q_1}=\Lambda_{q_2}$, then the solutions of Lemma \ref{2nd_lin_sol_free_bndr} corresponding to potentials $q_1$ and $q_2$ can be taken to have same boundary values. 
\begin{prop}\label{2nd_lin_sol_fixed_bndr}
Assume as in Lemma \ref{2nd_lin_sol_free_bndr} and adopt its notation, and assume that $\Lambda_{q_1}=\Lambda_{q_2}$ additionally. Then the second linearized equations 
 \begin{equation}\label{eq:second_lin_fixed_bndr_val_prop}
  (-\Delta+V)w^{(\beta)}=-2q_\beta v^{(1)}v^{(2)}, \quad \beta=1,2,
 \end{equation}
have solutions of the form
\[
 w^{(\beta)}=w_0^{(\beta)}+ e^{\tau \Psi}R^{(\beta)}.
\]
Here
\begin{align*}
 w_0^{(\beta)}&=\tau^{\frac{n-2}{4}}e^{(\pm s_1\pm s_2) x_1+\mathbf{i}(s_1 \psi_1+s_2 \psi_2)}b^{(a)}, \\
b^{(\beta)}&=\tau^{-2}b^{(a)}_{-2}+\cdots+ \tau^{-2N'}b^{(\beta)}_{-2N'}, \\
 b^{(\beta)}_{-2}&=\frac{2q_\beta}{(\pm c_1\pm c_2)^2-\abs{c_1\nabla'\psi_1+c_2\nabla'\psi_2}^2}a_0^{(1)}a_0^{(2)}.
\end{align*}
Moreover $R^{(\beta)}=\mathcal{O}_{L^2(M)}(\tau^{-N})$ ($\beta=1,2$) and 
\[
  w^{(1)}\big|_{\p M}= w^{(2)}\big|_{\p M}.
\]
\end{prop}

In order to prove Proposition \ref{2nd_lin_sol_fixed_bndr}, we need a \emph{boundary determination} result:

\begin{prop}[Boundary determination]\label{Prop_boundary_det_bodytext}
	For $m\geq 2$, $m\in \N$, let $(M,g)$ be a compact Riemannian manifold with $C^\infty$ boundary $\p M$ and consider the boundary value problem % for emilinear elliptic operator 
	\begin{align}\label{eq:boundary_det_equation}
		\begin{cases}
			(-\Delta_g +V) u+qu^m=0 & \text{ in } M,\\
			u= f & \text{ on } \p M.
		\end{cases}
	\end{align}
	Assume that the DN map $\Lambda_q$ of the equation \eqref{eq:boundary_det_equation} is known for small boundary values. Then $\Lambda_q$ determines the formal Taylor series of $q$ on the boundary $\p M$. 
	
	In addition, if $f\in C^{2,\alpha}(\p M)$ is so small that \eqref{eq:boundary_det_equation} has a unique small solution, the DN map determines the formal Taylor series of the solution $u=u_f$ at any point on the boundary.
\end{prop}

We also need the following Carleman estimate with boundary terms.

\begin{lem}[Carleman estimate with boundary terms]\label{lem0}
Let $(M,g)$ be a  compact, smooth, transversally anisotropic Riemannian manifold with a smooth boundary. Let $V \in L^{\infty}(M)$. There exists constants $\tau_0>0$ and $C>0$ depending only on $(M,g)$ and $\|V\|_{L^{\infty}(M)}$ such that given any $|\tau|>\tau_0$, and any $v \in C^2(M)$, there holds
\begin{align}\label{est_0} 
\begin{split}
C|\tau|\,\|v\|_{L^2(M)} \leq &\|e^{-\tau t}(-\Delta_g+V)(e^{\tau t}v)\|_{L^2(M)}+|\tau|^{\frac{3}{2}}\|v\|_{W^{2,\infty}(\p M)} \\
&+|\tau|^{\frac{3}{2}}\|\p_\nu v\|_{W^{1,\infty}(\p M)}+|\tau|^{\frac{3}{2}}\|\p^2_\nu v\|_{L^\infty(\p M)},
\end{split}
\end{align}
% \HOX{The estimate (0.1) is by no means sharp, but it is sufficient for our purposes.}
% for some $C>0$ that is independent of $\tau$.
\end{lem}
We have placed the proofs of the above two results in the the Appendix \ref{Section: Boundary determination} and \ref{Section: Boundary Carleman}, respectively. The proof of Proposition \ref{Prop_boundary_det_bodytext} uses a standard boundary determination result for linearized second order elliptic equations. The proof of Lemma \ref{lem0} is by integration by parts and using standard elliptic estimates. 
In this paper, the preceding Carleman estimate with the $L^2(M)$ bound is sufficient in deriving the upper bound
for the correction term $R^{(\beta)}$ in Proposition~\ref{2nd_lin_sol_fixed_bndr} for $\beta=1,2$; however let us also mention that analogous Carleman estimates with boundary terms can be obtained in higher Sobolev spaces $H^k(M)$, for $k\in \N$.

\begin{proof}[Proof of Proposition \ref{2nd_lin_sol_fixed_bndr}]
Let us first consider the case $\pm c_1\pm c_2\neq 0$. 
By Lemma \ref{2nd_lin_sol_free_bndr} we have a solution of the form 
\[
 w^{(2)}=w_0^{(2)}+e^{\tau\Psi} R^{(2)}
\]
for the equation
 \[
  (-\Delta+V)w^{(2)}=-2q_2v^{(1)}v^{(2)}.
 \]
In general, controlling the boundary value of $R^{(2)}$ is hard. As remarked in Remark \ref{rem:second_lin_lemma}, we have that $R^{(2)}$ is a smooth function defined on an open manifold $U$ such that $M\Subset  U$, which satisfies
\[
 \norm{R^{(2)}}_{H^{K}(U)}\leq \frac{C}{\tau^{N}}
\]
if $N'=N'(K,N)$ was chosen large enough.

By redefining $K$ as $K+5/2$ (and thus also redefining also $N'$ larger) and using trace theorem 
\begin{equation}\label{eq:dir_data_small}
 R^{(2)}|_{\p M}=\mathcal{O}_{H^K(\p M)}(\tau^{-N})
\end{equation}
and 
\begin{equation}\label{eq:neu_data_small}
 \left. \p_\nu R^{(2)}\right|_{\p M}=\mathcal{O}_{H^k(\p M)}(\tau^{-K}),   \qquad   \left. \p_\nu^2 R^{(2)}\right|_{\p M}=\mathcal{O}_{H^k(\p M)}(\tau^{-K}).
\end{equation}
%where we may choose $k$ and $K$ freely by making WKB ansatz for $W_2$ more precise.

Let us then consider the equation \eqref{eq:second_lin_fixed_bndr_val_prop} for $q_1$ with boundary value $\left. w^{(2)}\right|_{\p M}$. By the  standard elliptic theory, we know that there is a unique solution $w^{(1)}$ to the equation 
\begin{align}
 \begin{cases}
		(-\Delta+V)w^{(1)}=-2q_1v^{(1)}v^{(2)} &\text{ in } M,\\
		w^{(1)}=\left. w^{(2)}\right|_{\p M} &\text{ on }\p M.
	\end{cases}
\end{align}
We write
\[
 w^{(1)}=w_0^{(1)}+e^{\tau\Psi}R^{(1)},
\]
where $w_0^{(1)}=\tau^{\frac{n-2}{4}}e^{(\pm s_1\pm s_2) x_1+\mathbf{i}(s_1 \psi_1+s_1 \psi_1)}b_\tau^{(1)}$ is the WKB ansatz given as in Lemma \ref{2nd_lin_sol_free_bndr}  such that
\[
 (\Delta-V) w_0^{(1)}- 2q_1v^{(1)}v^{(2)}=e^{\tau\Psi}F.
\]
Here 
\begin{equation}\label{eq:F}
F=\mathcal{O}_{H^K(M)}(\tau^{-N}), 
\end{equation}
which can be derived by making the WKB ansatz $w_0^{(1)}$ precise enough (i.e. $N'$ large enough). Since $w^{(1)}$ solves $(-\Delta+V)w^{(1)}=-2q_1v^{(1)}v^{(2)}$, we have that $R^{(1)}$ solves the conjugated equation
\[
 e^{-\tau\Psi}(\Delta-V) e^{\tau\Psi}R^{(1)}=\mathcal{O}_{H^K(M)}(\tau^{-N}).
\]
Unfortunately, we can not directly deduce from standard Carleman estimates that the correction term $\norm{R^{(1)}}_{L^2(M)}$ is small. 

As matter of fact, in order to obtain that $\norm{R^{(1)}}_{L^2(M)}$ is small, we use the assumption $\Lambda_{q_1}=\Lambda_{q_2}$, which implies that the DN maps of the second linearized equations \eqref{2nd_lin_sol_fixed_bndr} for $q_1$ and $q_2$ are the same. By additionally using the boundary determination result (Proposition \ref{Prop_boundary_det_bodytext}), we have that 
\[
 q_1=q_2 \quad \text{ on }\p M
\]
up to infinite order. The ansatzes $w_0^{(1)}$ and $w_0^{(2)}$ depend on $(M,g)$ and the potentials $q_1$ and $q_2$ respectively. The dependence on the potentials is local. That is, the dependence is on pointwise values of the potentials and their derivatives, see \eqref{expansion_formula_new}. It follows that 
\begin{equation}\label{eq:w0k_bndr_val}
\left.  w_0^{(1)}\right|_{\p M}= \left. w_0^{(2)}\right|_{\p M} 
\end{equation}
and also 
\begin{equation}\label{eq:neumann_vals_of_ansatz}
\left. \p_\nu w_0^{(1)}\right|_{\p M}= \left. \p_\nu w_0^{(2)}\right|_{\p M} \quad \text{ and } \quad \left. \p_\nu^2 w_0^{(1)}\right|_{\p M}= \left. \p_\nu^2 w_0^{(2)}\right|_{\p M}.
\end{equation}
Consequently, by using $\left. w_1\right|_{\p M}=\left. w_2\right|_{\p M}$, we have that
\[
  \left. R^{(1)}\right|_{\p M}= \left. e^{-\tau\Psi}\right|_{\p M}(w^{(1)}-\left. w_0^{(1)})\right|_{\p M}=\left. e^{-\tau\Psi}\right|_{\p M} \left. (w^{(2)}-w_0^{(2)})\right|_{\p M}= \left. R^{(2)}\right|_{\p M}.
\]
By \eqref{eq:dir_data_small} we thus have that
\[
 R^{(1)}\big|_{\p M}=\mathcal{O}_{H^K(\p M)}(\tau^{-N}).
\]

Furthermore,  we have  $\left. \p_\nu w_1 \right|_{\p M}=\left. \p_\nu w_2\right|_{\p M}$ since $\Lambda_{q_1}=\Lambda_{q_2}$. Consequently, by \eqref{eq:neumann_vals_of_ansatz} we have
\begin{align*}
 \p_\nu R^{(1)}\big|_{\p M}&=\p_\nu\left(e^{-\tau\Psi}(w^{(1)}-w_0^{(1)})\right)\Big|_{\p M}=\p_\nu \left(e^{-\tau\Psi}(w^{(2)}-w_0^{(2)})\right)\Big|_{\p M} \\
 &= \p_\nu R^{(2)}\big|_{\p M}=\mathcal{O}_{H^K(\p M)}(\tau^{-N}).
\end{align*}
By the boundary determination result of solutions on the boundary in Proposition \ref{Prop_boundary_det_bodytext}, we have $\left. \p^2_\nu w_1\right|_{\p M}= \left. \p^2_\nu w_2\right|_{\p M}$. Thus, combining \eqref{eq:neu_data_small} and \eqref{eq:neumann_vals_of_ansatz} shows $\left. \p_\nu^2 R^{(1)}\right|_{\p M}=\mathcal{O}_{H^K(\p M)}(\tau^{-N})$. 
In conclusion, we have that $R^{(1)}$ solves 
\begin{align}\label{eq:equation_for_r1}
 \begin{cases}
		e^{-\tau\Psi}(\Delta-V)e^{\tau\Psi} R^{(1)} = \mathcal{O}_{H^K(M)}(\tau^{-N}) &\text{ in } M,\\
		\p_\nu^{\s \ell} R^{(1)}=\mathcal{O}_{H^K(\p M)}(\tau^{-N}) &\text{ on }\p M, \quad \ell=0,1,2.
	\end{cases}
\end{align}
Now, it follows from Lemma \ref{lem0} by taking $K=\frac{n+1}{2}$ and using the Sobolev embedding $H^{K}(\p M)\subset L^{\infty}(\p M)$, and finally redefining $N$ as $N-2$ that 
\[
 \norm{R^{(1)}}_{L^2(M)}=\mathcal{O}(\tau^{-N}).
\]

In the remaining case $\pm c_1\pm c_2= 0$, the correction terms $R^{(1)}$ and $R^{(2)}$ have zero boundary values by Remark \ref{rem:second_lin_lemma}. Since we also have $w_0^{(1)}|_{\p M}=w_0^{(2)}|_{\p M}$ by \eqref{eq:w0k_bndr_val}, the claim follows also in this case.
\end{proof}
\subsection{Solutions for the third linearization}
In this section, we consider solutions for the third linearizations of $(-\Delta+V)u+qu^2=0$ at the zero solution. Recalling that the third linearized equation is of the form
\begin{align}\label{w_eq 3rd_constr_sols}
		(-\Delta_{g}+V)\omega^{(ijk)}=-2q\left( v^{(i)} w^{(jk)} + v^{(j)}w^{(ik)} + v^{(3=k)} w^{(ij)}\right) 
		&\text{ in }M,
\end{align}
where $v^{(i)}$ and $w^{(kl)}$, are solutions to  \eqref{w_eq} and \eqref{w_eq 3rd}, respectively, for different $i,j,k=1,2,3$. Again, we assume that the solutions $v^{(k)}$ are CGOs of the form \eqref{cgo_sol_1}:
\[
 v^{(k)} = e^{\pm s_k x_1}\left( \tau^{\frac{n-2}{8}} e^{\mathbf{i}s_k \psi_k} a^{(k)}_\tau + r_\tau^{(k)}\right),
\]
where $\psi_k$ corresponds to a nontangential geodesics $\gamma_k$ of $(M_0,g_0)$. Here $s_k=c_k\tau+\mathbf{i}\lambda_k$ also as before. We assume that $\gamma_1$, $\gamma_2$ and $\gamma_3$ intersect at the point $p_0$. We also assume that the supports of $v^{(k)}$ restricted to $M_0$ are so small that that the mutual support of $v^{(1)}$, $v^{(2)}$ and $v^{(3)}$ does not intersect the points on the geodesics $\gamma_k$ where only two of the geodesics intersect.
%in a small enough neighborhoods of the geodesics $\gamma_i$ so that the mutual support of $v^{(i)}$ does not intersect points where only two of the geodesics $\gamma_i$ intersect.  supported only on neighborhoods of the points where all the three geodesics $\gamma_i$ intersect. 
%Especially $v^{(i)}v^{(k)}v^{(l)}$ has support on a neighborhood of $p_0$. 
Lastly, we assume that all the pairs of geodesics $\gamma_i$ and $\gamma_k$, $i\neq k$, intersect properly. 

In order to analyze the solution ansatz for the third linearized equation \eqref{w_eq 3rd_constr_sols}, we can simply consider the case $i=1$, $j=2$ and $k=3$.
By Lemma \ref{2nd_lin_sol_free_bndr}, the equation $(-\Delta +V)w^{(23)}=-qv^{(2)}v^{(3)}$ has a solution of the form
\[
 w^{(23)}=w_0^{(23)}+e^{\tau\Psi^{(23)}} R^{(23)}.
\]
Here $w_0^{(23)}$ is given by the WKB ansatz
\begin{align*}
 w_0^{(23)}&=\tau^{\frac{n-2}{4}}e^{(\pm s_2\pm s_3) x_1+\mathbf{i}(s_2 \psi_2+s_3 \psi_3)}b^{(23)}, \\
b^{(23)}&=\tau^{-2}b^{(23)}_{-2}+\cdots+ \tau^{-2N}b^{(23)}_{-2N'}, \\
 b^{(23)}_{-2}&=\frac{2q}{(\pm c_2\pm c_3)^2-\abs{c_2\nabla'\psi_2+c_3\nabla'\psi_3}^2}a_0^{(2)}a_0^{(3)}.
\end{align*}
We take the solutions $w^{(13)}$ and $w^{(12)}$ to be ones given by similar formulas as $w^{(23)}$.  Using these formulas for $w^{(ik)}$ and $v^{(j)}$ we see that \eqref{w_eq 3rd_constr_sols} can be written as
	\[
	 (\Delta_{g}-V)\omega=\tau^{\frac{3(n-2)}{8}}e^{\tau \widetilde\Psi}(e^{\widetilde\Lambda} H+\rho),
	\]
	where $\omega\equiv \omega^{(123)}$, and 
	\begin{align}\label{eq:tildePsi}
	 \begin{split}
	 \widetilde\Psi&=(\pm c_1\pm c_2\pm c_3)x_1+\mathbf{i}c_1 \psi_1+\mathbf{i}c_2 \psi_2+\mathbf{i}c_3 \psi_3 \\
	 \widetilde\Lambda&=\mathbf{i}(\pm\lambda_1\pm \lambda_2\pm \lambda_3)x_1-\lambda_1 \psi_1-\lambda_2 \psi_2-\lambda_3 \psi_3 \\
	 H&=\sum_{k=2}^{3N'-1} \frac{H_{-k}}{\tau^{k}}, \\
	 \rho&=\mathcal{O}_{H^K(M)}(\tau^{-N}).
	 \end{split}
	 \end{align}
% 	 We assume that $\gamma_1$, $\gamma_2$ and $\gamma_3$ intersect at the point $p_0$. We also assume that the supports of $v^{(k)}$ restricted to $M_0$ are so small that that the mutual support of $v^{(1)}$, $v^{(2)}$ and $v^{(3)}$ does not intersect the points on the geodesics $\gamma_k$ where only two or less of the geodesics intersect.
	The amplitude $H\in C^\infty(M)$ is supported on neighborhoods of the points where all the geodesics $\gamma_1$, $\gamma_2$ and $\gamma_3$ intersect and which do not contain points where only two of the geodesics $\gamma_k$ intersect. The order $3N'-1$ of the amplitude $H$ is a consequence of the respective orders $2N'$ and $N'-1$ of the expansions of $w^{(ij)}$ and $a^{(k)}$. We have also assumed  $N'$ to be large enough so that the condition for $\rho$ in \eqref{eq:tildePsi} holds. Meanwhile, the factor $\tau^{\frac{3(n-2)}{8}}$ is a result of the product of the respective normalization factors $\tau^{\frac{n-2}{4}}$ and $\tau^{\frac{n-2}{8}}$ of $w^{(ij)}$ and $v^{(k)}$.  
	The functions $H_{-k}$ depend on $q$ only in terms of the pointwise values $q$ and its derivatives.

	By \eqref{expansion_formula_new}, the leading order coefficient of $H$ satisfies 
	\begin{align}\label{eq:H0_comp_orig}
   \begin{split}
   	 H_{-2}=&4q^2a_0^{(1)}a_0^{(2)}a_0^{(3)} \\
   	 &\times \Bigg(\frac{1}{(\pm c_1\pm c_2)^2-\abs{c_1\nabla'\psi_1+c_2\nabla'\psi_2}^2}  \\
   	 &\qquad  +\frac{1}{(\pm c_1\pm c_3)^2-\abs{c_1\nabla'\psi_1+c_3\nabla'\psi_3}^2} \\
   	& \qquad +\frac{1}{(\pm c_2\pm c_3)^2-\abs{c_2\nabla'\psi_2+c_3\nabla'\psi_3}^2}\Bigg). 
   \end{split}
	\end{align}
If we additionally assume that 
	\[
	 \abs{\langle \nabla \widetilde\Psi, \nabla \widetilde\Psi\rangle}\geq \text{ constant } > 0 
	\]
    on the support of $H$, 
    %neighborhoods of the points where all the  geodesics $\gamma_1$, $\gamma_2$ and $\gamma_3$ intersect, 
    it makes sense to try an ansatz
    \[
     \tau^{\frac{3(n-2)}{8}}e^{\tau \widetilde\Psi}e^{\widetilde\Lambda}B
    \]
    for a solution $\omega$ of \eqref{w_eq 3rd_constr_sols}, where
    \begin{equation}\label{expansion_formula_new_B}
    B=\sum_{k=4}^{3N'+1} \frac{B_{-k}}{\tau^k}
    \end{equation}
    Here $B_{-k}$, $k=4,3,\ldots, 2(N'+2)$ are given by the recursive formula
    \begin{align}\label{transport_equation_new_B}
		B_{-k}= \frac{e^{\widetilde\Lambda} H_{-k+2}-[2\langle\nabla  B_{-k+1}, \nabla \Psi\rangle+B_{-k+1}\Delta\Psi] 
			-(\Delta-V)B_{-k+2}}{\langle \nabla \widetilde \Psi, \nabla \widetilde\Psi\rangle}
	\end{align}
	and setting $B_{-2}=B_{-3}=0$. 
    Especially
    \begin{align}\label{eq:B_2}
B_{-4}=\frac{H_{-2}}{\langle \nabla \widetilde\Psi, \nabla \widetilde\Psi\rangle },
\end{align}
where $H_{-2}$ is given in \eqref{eq:H0_comp_orig}. The support of $B$ is the mutual support of $v^{(k)}$.

We obtain the following result. We omit the proof as it is a direct adaptation of the proof of Lemma \ref{2nd_lin_sol_free_bndr}.
\begin{lem}\label{3rd_lin_sol_free_bndr}
 Let $K,N\in \N\cup \{0\}$. Assume that $v^{(1)}, v^{(2)}, v^{(3)}$ are CGOs of the form \eqref{cgo_sol_1} corresponding to geodesics $\gamma_1, \gamma_2, \gamma_3$ on $M_0$, respectively, such that the pairs of geodesics $\gamma_k$ and $\gamma_i$  intersect properly for $i,k=1,2,3$ and $i\neq k$.  Assume additionally that $\widetilde\Psi$ given by \eqref{eq:tildePsi} satisfies
 \[
  \langle \nabla\widetilde \Psi, \nabla\widetilde \Psi\rangle\neq 0
 \]
at the points where all the geodesics $\gamma_1$, $\gamma_2$ and $\gamma_3$ intersect. 
If the restrictions of the amplitudes $a^{(k)}$ of $v^{(k)}$ to $M_0$ are supported in small enough neighborhoods of the geodesics $\gamma_k$, and $N'=N'(K,N)$ is large enough, then the equation 
\begin{align}\label{eq:3nd_lin_sol_fixed_bndr}
(-\Delta_{g}+V)\omega=-2q\left( v^{(1)} w^{(23)} + v^{(2)}w^{(13)} + v^{(3)} w^{(12)}\right) 
		&\text{ in }M,
    \end{align}
where $w^{(ik)}$ is given as in Lemma \ref{2nd_lin_sol_free_bndr} has a smooth solution $\omega$ up to the boundary $\p M$ with the following properties: The solution $w$ is of the form
\[
 \omega=\omega_0+e^{\tau \widetilde\Psi}\widetilde R,
\]
where the function $\omega_0$ is of the form 
\[
	\omega_0=\tau^{\frac{3(n-2)}{8}}e^{\tau \widetilde\Psi}e^{\widetilde\Lambda}B,
	\]
    where $\widetilde\Lambda$ and $B=B_\tau$ are given by \eqref{eq:tildePsi} and \eqref{expansion_formula_new_B} respectively. Especially $B_{-4}$ is given by \eqref{eq:B_2}. 
% 	\[
% 	B_{-2}=\frac{q}{\langle \nabla \widetilde\Psi, \nabla \widetilde\Psi\rangle }H_0.
% 	%\frac{q}{(\pm c_1\pm c_2)^2-\abs{c_1\nabla'\psi_1+c_2\nabla'\psi_2}^2}a_0^{(1)}a_0^{(2)}.
% 	\]
	The amplitude $B$ depends on $q$ only in terms of the pointwise values $q$ and its derivatives. The remainder term 
	$\widetilde R=\widetilde R_\tau$ satisfies
	\[
\norm{\widetilde R_\tau}_{H^{K}(M)}\lesssim \tau^{-N}.
	\]
\end{lem}
As stated, the amplitude $B$ depends on $q$ only in terms of the pointwise values $q$ and its derivatives. Thus, by assuming that we know the DN map of $(-\Delta+V)u + qu^2=0$, we may determine the value of $\omega_0$ on the boundary by  boundary determination result (Proposition \ref{Prop_boundary_det_bodytext}). Consequently, by using the Carleman estimate with boundary terms (Lemma \ref{lem0}), we have the following analogous result of Proposition \ref{2nd_lin_sol_fixed_bndr}. Note that 
$$
{\langle \nabla \widetilde\Psi, \nabla \widetilde\Psi\rangle }=(\pm c_1\pm c_2\pm c_3)^2-\abs{c_1\nabla'\psi_1+c_2\nabla'\psi_2+c_3\nabla'\psi_3}^2.
$$

\begin{prop}\label{3rd_lin_sol_fixed_bndr}
Assume as in Lemma \ref{3rd_lin_sol_free_bndr} and adopt its notation. Assume additionally that $\Lambda_{q_1}=\Lambda_{q_2}$. If the restrictions of the amplitudes $a^{(k)}$ of $v^{(k)}$ to $M_0$ are supported in small enough neighborhoods of the geodesics $\gamma_k$, and $N'=N'(K,N)$ is large enough, then  the third linearized equations 
\begin{align}%\label{eq:3nd_lin_sol_fixed_bndr}
(-\Delta_{g}+V)\omega^{(\beta)}=-2q_\beta\left( v^{(1)} w_\beta^{(23)} + v^{(2)}w_\beta^{(13)} + v^{(3)} w_\beta^{(12)}\right) 
		&\text{ in }M,
    \end{align}
    where $w_\beta^{(ik)}$, for $\beta=1,2$ and different $i,k=1,2,3$, are given as in Proposition \ref{2nd_lin_sol_fixed_bndr}. Moreover, the solution $\omega^{(\beta)}$ is of the form 
\[
 \omega^{(\beta)}=\omega_0^{(\beta)}+ e^{\tau \widetilde\Psi}\widetilde R^{(\beta)}
 ,
\]
where
\begin{align*}
\omega_0^{(\beta)}&=\tau^{\frac{3(n-2)}{8}}e^{\tau \widetilde\Psi}e^{\widetilde\Lambda}B^{(\beta)}, \\
B^{(\beta)}&=\tau^{-4}B^{(\beta)}_{-4}+\cdots+ \tau^{-3N'+1}B^{(\beta)}_{-3N'+1}.
\end{align*}
Here $\widetilde\Lambda$ and $\widetilde\Psi$ are given by \eqref{eq:tildePsi}. 
Especially, the quantity $B_{-4}$ in \eqref{eq:B_2} can be written as 
\begin{align}\label{eq:H0_comp}
   \begin{split}
   	 B^{(\beta)}_{-4}=&4q_\beta^3\s a_0^{(1)}a_0^{(2)}a_0^{(3)}\frac{1}{(\pm c_1\pm c_2\pm c_3)^2-\abs{c_1\nabla'\psi_1+c_2\nabla'\psi_2+c_3\nabla'\psi_3}^2}\\
   	&\quad \times \Bigg(\frac{1}{(\pm c_1\pm c_2)^2-\abs{c_1\nabla'\psi_1+c_2\nabla'\psi_2}^2} \\
   	&\quad \qquad +\frac{1}{(\pm c_1\pm c_3)^2-\abs{c_1\nabla'\psi_1+c_3\nabla'\psi_3}^2} \\
   	&\quad \qquad +\frac{1}{(\pm c_2\pm c_3)^2-\abs{c_2\nabla'\psi_2+c_3\nabla'\psi_3}^2}\Bigg). 
   \end{split}
	\end{align}
and $\widetilde R^{(\beta)}=\mathcal{O}_{L^2(M)}(\tau^{-N})$, for $\beta=1,2$. Moreover 
\[
  \omega^{(1)}\big|_{\p M}= \omega^{(2)}\big|_{\p M}.
\]
\end{prop}
We skip the proof of Proposition \ref{3rd_lin_sol_fixed_bndr} as it can be obtained from the proof of Proposition \ref{2nd_lin_sol_fixed_bndr} by replacing $w$ by $\omega$ and $\Psi$ by $\widetilde\Psi$ etc. The function $F$ in \eqref{eq:F} in the proof also needs to be replaced by a function of the class $\mathcal{O}_{L^2(M)}(\tau^{-N})$ since $R^{(2)}$ in Proposition \ref{2nd_lin_sol_fixed_bndr} is $\mathcal{O}_{L^2(M)}(\tau^{-N})$. We remark that by deriving Carleman estimates similar to those in Lemma \ref{lem0} for higher Sobolev spaces, we could in fact have that $\widetilde R$ is of the size $\tau^{-N}$ also in higher Sobolev spaces $H^K(M)$ by taking $N'$ large enough.
% 
% ... and $F$ in \eqref{eq:F} by a function of the class $\mathcal{O}_{L^2(M)}(\tau^{-N})$. The need for the latter replacement is a consequence of the fact that the correction terms $R^{(1)}$ and $R^{(2)}$ in Proposition \ref{2nd_lin_sol_fixed_bndr} are $\mathcal{O}_{L^2(M)}(\tau^{-N})$.

\section{Proof of Theorem \ref{t1}}
\label{Sec:proof_of_t1}
In this section we prove Theorem \ref{t1}.  We will see that it is possible to deduce 
\[
q_1^2=q_2^2 
\]
in $M$ from third order linearizations and the DN map of the equation $(-\Delta +V)u+qu^2=0$. Our method for the third linearized equation however does not imply $q_1=q_2$ in general.  To show that 
\[
q_1=q_2 
\]
we in fact need to consider fourth order linearized equations. To give a proof of Theorem \ref{t1}, we could consider the fourth order linearization from the beginning. However, we first consider third order linearizations and prove $q_1^2=q_2^2$ to better explain the main ideas of the proof.

\subsection{Proof of $q_1^2=q_2^2$} Let $p_0\in M_0$, and let $\gamma_1$ be a non-tangential geodesic that has no self-intersections. 
% We start by proving $q_1^2=q_2^2$ in a simplified setting where we assume that geodesics on $(M_0,g_0)$ have no self-intersections and that pairs of geodesics can intersect only once.\f{Let's decide if we want to allow geodesics to intersect.}  
We consider the equation
\begin{align}\label{equ_integral id0_proof_Sec}
	\begin{cases}
		(-\Delta +V)u_a +q_a u_a^2=0 & \text{ in }M,\\
		u_a =f &\text{ on }\p M,
	\end{cases}
\end{align}
for $a=1,2$, where $f=f_\eps\in C^{2,\alpha}(\p M)$ is of the form 
\begin{align}\label{f_epsilon_1}
	f_\eps:=\displaystyle \sum_{i=1}^4\eps_i f_i+\sum_{i,j =1}^4 \eps_i \eps_j f_{ij}\quad \text{ on }\quad \p M.
\end{align}

Let us recall the linearizations \eqref{equ_integral id0_proof_Sec} from Section \ref{sec:higher_ord_lin_with_bndr_vals}. The first linearization reads 
\begin{align}\label{equ_integral first id0}
\begin{cases}
		(-\Delta + V)v_\beta^{(i)}=0 & \text{ in }M,\\
		v_\beta^{(i)}=f_i    & \text{ on }\p M,
\end{cases}
\end{align}
where $v_\beta^{(i)}=\left. \p_{\eps_i} \right|_{\eps=0} u_\beta$ for $\beta=1,2$, and $i=1,2,3,4$.
By the uniqueness of solutions to \eqref{equ_integral first id0}, we obtain 
\[
v^{(i)}:=v_1^{(i)}=v_2^{(i)} \text{ in }M, 
\]
for $i=1,2,3,4$. The second linearization of \eqref{equ_integral id0_proof_Sec} satisfies 
\begin{align}\label{equ_integral second id0}
	\begin{cases}
		(-\Delta +V)w_\beta^{(jj)} =-2q_\beta v^{(j)}v^{(j)}& \text{ in }M,\\
		w_\beta^{(jj)}=f_{jj} &\text{ on }\p M,
	\end{cases}
\end{align}
where 
\[
w_\beta^{(jj)}=\left. \p^2 _{\eps_j \eps_j}\right|_{\eps=0} u_\beta, 
\]
 for $\beta=1,2$ and different $i,j\in \{1,2,3\}$. Lastly, the third linearization of \eqref{equ_integral id0_proof_Sec} satisfies
\begin{align}\label{equ_integral third id0}
	\begin{cases}
		(-\Delta +V)w_\beta^{(ijk)} =-2q_\beta \left( v^{(i)} w_\beta^{(jk)} + v^{(j)}w_\beta^{(ik)} + v^{(k)} w_\beta^{(ij)}\right)  & \text{ in }M,\\
		w_\beta^{(ikl)} =0 &\text{ on }\p M,
	\end{cases}
\end{align}
where 
\[
w_\beta^{(ijk)}=\left. \p^3 _{\eps_i \eps_j\eps_k}\right|_{\eps=0} u_\beta.
\]
 %for $a=1,2$ and $i,j,k =1,2,3,4$. 
Since $\Lambda_{q_1}=\Lambda_{q_2}$%, we have %third order derivatives of DN maps satisfy
 \begin{align}\label{equal DN with derivatives}
 	\begin{split}
 		%&0=\left.\p^2_{\epsilon_i\epsilon_j} \right|_{\epsilon=0} \left( \Lambda_{q_1}  -\Lambda_{q_2}\right)( f_\epsilon)=
 		&0=\left. \p^3_{\eps_i \eps_j \eps_k}\right|_{\eps=0}\left( \Lambda_{q_1}-\Lambda_{q_2}\right)(f_\eps) \\
 		%=&\left.\p^4 _{\eps_1 \eps_2 \eps_3\eps_4}\right|_{\epsilon=0} \left(   \Lambda_{q_1} f_\epsilon- \Lambda_{q_2} f_\epsilon \right) =0,
 	\end{split}
 \end{align}
 Thus, by Lemma \ref{Lem:Integral identities} we have  
	\begin{align}\label{integral id for 3rd linearization}
		\begin{split}
		&0=\int_{M} \left\{ q_1\left( v^{(i)} w_1^{(jk)} + v^{(j)}w_1^{(ik)} + v^{(k)} w_1^{(ij)}\right)  \right. \\
			       & \qquad \qquad \s \left. - q_2\left( v^{(i)} w_2^{(jk)} + v^{(j)}w_2^{(ik)} + v^{(k)} w_2^{(ij)}\right)  \right\} v^{(l)}\, dV,
		\end{split}
	\end{align}
where $v^{(i)}$ and $w_\beta^{(jk)}$ are the solutions of \eqref{equ_integral first id0} and \eqref{equ_integral second id0}, respectively, for different $i,j,k=1,2,3,4$ and $\beta=1,2$.

We choose $v^{(i)}$  to be CGOs corresponding geodesics on $(M_0,g_0)$, which intersect properly pairwise. We show that the integrand on the right hand side of \eqref{integral id for 3rd linearization} restricted to a neighborhood of $p_0$ in $M_0$ is close to a multiple of the delta function. We let $v^{(1)}$ correspond to the geodesic $\gamma_1$ and choose the other $3$ geodesics next. 
% 
% Recalling that we have derived integral identities for the second, third and fourth order linearized equations in Lemma \ref{Lem:Integral identities}.
% Next, by  assuming $\Lambda_{q_1}(f_\eps)=\Lambda_{q_2}(f_\eps)$ on $\p M$, the integral identity \eqref{third integral id} of course vanishes. We choose $v^{(k)}$, $k=1,2,3,4$, to be CGOs corresponding to pairs of properly intersecting geodesics on $(M_0,g_0)$. In this case, we show that the integrand on the right hand side of \eqref{integral id for 3rd linearization} is close to a delta function (in the current simplified setting). For this, let $p_0\in M_0$, and let $\gamma_1$ be a non-tangential geodesic with initial data $\xi_1\in S_{p_0}M_0$. We choose the other $3$ geodesics next. 

\subsection{Choices of initial vectors for the third linearization}\label{sec:choice_of_vecs_3rd_lin}
Let $\delta\in (0,1)$, and we denote the initial data of $\gamma_1$ by $\xi_1\in S_{p_0}M_0$. By perturbing $\xi_1$, we find $\xi_2\in S_{p_0}M_0$ such that the associated geodesic $\gamma_2$ is also non-tangential, has no self-intersections, and that
\[
\abs{\xi_1}=\abs{\xi_2}=1
\]
and 
\[
\langle \xi_1 , \xi_2 \rangle =1-\delta.
\]
Let us define
\[
\xi_3=-\frac{1}{1+\delta}(\xi_1+\delta \xi_2)\in S_{p_0}M_0 \quad \text{ and } \quad \xi_4=-\frac{1}{1+\delta}(\delta\xi_1 + \xi_2)\in S_{p_0}M_0.
\]
A direct computation shows 
\begin{align}\label{eq:real_part_of_sum_xi}
	\sum_{l=1}^4 \xi_l&=\xi_1+\xi_2-\frac{1}{1+\delta}\xi_1-\frac{\delta}{1+\delta} \xi_2-\frac{\delta}{1+\delta}\xi_1-\frac{1}{1+\delta} \xi_2 =0.
\end{align}
We redefine $\delta$ smaller, if necessary, so that the geodesics $\gamma_3$ and $\gamma_4$ corresponding to $\xi_3$ and $\xi_4$ are also nontangential.

Note that $\xi_1$ is not proportional to $\xi_2$ as $\xi_1$ and $\xi_2$ are linearly independent. Similarly, for $k=3,4$, the vector $\xi_k$ is neither proportional to $\xi_1$ nor to $\xi_2$. Lastly, $\xi_3$ is not proportional to $\xi_4$. Indeed, if $A\in \R$ is such that $\xi_3=A \xi_4$, we have that $1=\delta A$ and $\delta = A$, implying that $\delta =\pm 1$. However, $\delta\in (0,1)$. This means that all the pairs of the geodesics corresponding to initial data $\xi_k/\abs{\xi_k}$ intersect properly. 
%In this simplied setting they intersect only at $p_0$.

Note also that since $\abs{\xi_1}=\abs{\xi_2}=1$, we have
\begin{align*}
	\abs{\xi_3}^2=&\frac{1}{(1+\delta)^2}\left(\abs{\xi_1}^2+\delta^2\abs{\xi_2}^2+2\delta \langle \xi_1, \xi_2\rangle\right)\\
	=&\frac{1}{(1+\delta)^2}\left(\abs{\xi_2}^2+\delta^2\abs{\xi_1}^2+2\delta \langle\xi_2, \xi_1\rangle\right)
	=\abs{\xi_4}^2.
\end{align*}
That is 
\begin{equation}\label{eq:norm_xi3_xi4}
 \abs{\xi_3}^2=\abs{\xi_4}^2=\frac{1}{(1+\delta)^2}\left(1+\delta^2+2\delta (1-\delta)\right)=\frac{1}{(1+\delta)^2}\left(1+2\delta -\delta^2\right)
\end{equation}

%(This will result in cancellation of the exponential factor in the product of four solutions.)

Let us then define vectors $\overline{\xi}_k\in TM$, $k=1,2,3,4$, by
\begin{align}\label{eq:vecs_olxia}
\begin{split}
 \overline{\xi}_1&=\abs{\xi_1}e_1+\mathbf{i}\xi_1, \quad \overline{\xi}_2=-\abs{\xi_2}e_1+\mathbf{i}\xi_2 \\
 \overline{\xi}_3&=\abs{\xi_3}e_1+\mathbf{i}\xi_3, \quad \overline{\xi}_4=-\abs{\xi_4}e_1+\mathbf{i}\xi_4.
 \end{split}
\end{align}
Then 
\begin{equation}\label{eq:lightlike_sums_to_zero}
 \sum_{k=1}^4\overline{\xi}_k=0.
\end{equation}
Note also that
\begin{equation}\label{eq:xibar_lighlike}
 \langle \overline{\xi}_k,\overline{\xi}_k\rangle=0, \quad k=1,\ldots,4.
\end{equation}
Related to these vectors $\overline{\xi}_k$, we will consider in the proof of Theorem \ref{t1}  CGOs, which can be written of the form
\begin{align*}
 v_{s}^{(k)}&= e^{\text{Re}(\overline{\xi}_k)  x_1} \left(\tau^{\frac{n-2}{8}}e^{\mathbf{i} \abs{\xi_k}\psi_k} a_s + r_s \right).
\end{align*}
%where $\psi_k$, $k=1,2,3,4$, are phase functions corresponding to initial vector $\xi_k/\abs{\xi_k}$
Here the phase functions $\psi_k$ are constructed with respect to the geodesics $\gamma_k$ with initial data $\gamma_k(0)=p_0$ and $\dot \gamma_k(0)=\frac{\xi_k}{\abs{\xi_k}}$. We note  that 
\begin{equation}\label{eq:nabla_on_geo}
 \nabla \Big(\text{Re}(\overline{\xi}_k)  x_1+\mathbf{i} \abs{\xi_k}\psi_k)\Big)\Big|_{\gamma_k(0)}=\overline \xi_k.
\end{equation}
Consequently, the ansatzes $w_0^{(ik)}$ in Lemma \ref{2nd_lin_sol_free_bndr} for the solutions of $(-\Delta +V)w^{(ik)}=-2qv^{(i)}v^{(k)}$ have amplitudes with a factor that divides by
\begin{align*}
 &\left\langle\nabla \big(\text{Re}(\overline{\xi}_i+\overline{\xi}_k)  x_1+\mathbf{i}( \abs{\xi_i}\psi_i+ \abs{\xi_k}\psi_k)\big),\nabla \big(\text{Re}(\overline{\xi}_i+\overline{\xi}_k)  x_1+\mathbf{i}( \abs{\xi_i}\psi_i+ \abs{\xi_k}\psi_k)\big)\right\rangle,
\end{align*}
for different $i,k=1,2,3$. 
At an intersection point of the geodesics $\gamma_i$ and $\gamma_k$ the above equals
\begin{align*}
 &\langle \overline{\xi}_i+\overline{\xi}_k,\overline{\xi}_i+\overline{\xi}_k\rangle =2\left\langle\nabla \big(\text{Re}(\overline{\xi}_i)  x_1+\mathbf{i} \abs{\xi_i}\psi_i\big),\nabla \big(\text{Re}(\overline{\xi}_k)  x_1+\mathbf{i} \abs{\xi_k}\psi_k\big)\right\rangle = 2\left\langle \overline \xi_i, \overline \xi_k\right\rangle.
\end{align*}
Here we used \eqref{eq:nabla_on_geo} and \eqref{eq:xibar_lighlike}. Motivated by this, we define
\[
 \mathbf{C}_{ik}:=2\left\langle \overline \xi_i, \overline \xi_k\right\rangle.
\]
The coefficient $\mathbf{C}_{ik}$ can be collectively  written as
\[
 \mathbf{C}_{ik}=2\abs{\xi_i}\abs{\xi_k}\left((-1)^{i+k}-\frac{\left\langle \xi_i, \xi_k\right\rangle}{\abs{\xi_i}\abs{\xi_k}}\right).
\]

We calculate expansions for $\mathbf{C}_{ik}$ for small $\delta>0$ parameter. A direct computation shows that
\[
\langle \xi_1 ,\xi_3 \rangle=\left\langle \xi_1, -\frac{1}{1+\delta}(\xi_1+\delta \xi_2)\right\rangle=-\frac{1}{1+\delta}(\abs{\xi_1}^2+\delta \langle\xi_1, \xi_2\rangle)=-\frac{1}{1+\delta}(1+\delta  -\delta^2)
\]
and
\[
\langle  \xi_2 ,\xi_3\rangle=-\frac{1}{1+\delta}\left(\langle \xi_1,\xi_2 \rangle+\delta |\xi_2|^2\right)=-\frac{1}{1+\delta},
\]
where we used $\langle \xi_1,\xi_2\rangle=1-\delta$.
We also have
\[
 %\nabla'\psi_1\cdot \nabla'\psi_3=
 \frac{\langle \xi_1 ,\xi_3\rangle}{\abs{\xi_1}\abs{\xi_3}}=\frac{-\frac{1}{1+\delta}(1+\delta  -\delta^2)}{\left(\frac{1}{(1+\delta)^2}\left(1+2\delta -\delta^2\right)\right)^{1/2}}=-1+\mathcal{O}(\delta)
\]
and
\[
 %\nabla'\psi_2\cdot \nabla'\psi_3=
 \frac{\langle \xi_2 ,\xi_3\rangle}{\abs{\xi_2}\abs{\xi_3}}=\frac{-\frac{1}{1+\delta}}{\left(\frac{1}{(1+\delta)^2}\left(1+2\delta -\delta^2\right)\right)^{1/2}}=-1+\mathcal{O}(\delta).
\]
Here we have utilized the Taylor expansions
\[
(1+r)^{1/2}=1+r/2+\mathcal{O}(r^2) \quad  \text{ and } \quad  (1+r)^{-1}=1-r+\mathcal{O}(r^2),
\]
which hold  for small $|r|$. Combining the above formulas yields
\begin{align}\label{C_12, C_23, C_13}
	\begin{split}
		\mathbf{C}_{12}&=2\abs{\xi_1}\abs{\xi_2}\left(-1-\frac{\langle \xi_1 ,\xi_2\rangle}{\abs{\xi_1}\abs{\xi_2}}\right)=-4+\mathcal{O}(\delta),\\
			\mathbf{C}_{13}&=2\abs{\xi_1}\abs{\xi_3}\left(1-\frac{\langle \xi_1 ,\xi_3\rangle}{\abs{\xi_1}\abs{\xi_3}}\right)=2\frac{\left(1+2\delta -\delta^2\right)^{1/2}}{1+\delta}\Big(1-\big(-1+\mathcal{O}(\delta)\big)\Big) \\
			&=4+\mathcal{O}(\delta), \\
			\mathbf{C}_{23}&=2\abs{\xi_2}\abs{\xi_3}\left(-1-\frac{\langle \xi_2 ,\xi_3\rangle}{\abs{\xi_2}\abs{\xi_3}}\right) =2\frac{\left(1+2\delta -\delta^2\right)^{1/2}}{1+\delta}\Big(-1-\big(-1+\mathcal{O}(\delta)\big)\Big)\\
			&=\mathcal{O}(\delta).
	\end{split}
\end{align}
We also remark here that $\mathbf{C}_{ik}\neq 0$, $i\neq k$, for $\delta>0$ since 
\begin{equation}\label{eq:Cik_not_zero}
 \abs{\mathbf{C}_{ik}}=2\abs{\xi_i}\abs{\xi_k}\left|(-1)^{i+k}-\frac{\langle \xi_i ,\xi_i\rangle}{\abs{\xi_i}\abs{\xi_k}}\right|
\end{equation}
and 
\[
 \frac{\langle \xi_i ,\xi_k\rangle}{\abs{\xi_2}\abs{\xi_3}}\in (-1,1),
\]
because the pairs of vectors $\xi_i$ and $\xi_k$ are linearly independent. Finally, we note that 
\begin{align*}\label{inverse summation of third linearization}
	\begin{split}
		\left| \frac{1}{\mathbf{C}_{12}}+\frac{1}{\mathbf{C}_{13}}+	\frac{1}{\mathbf{C}_{13}}\right| =\left|\frac{1}{-4+\mathcal{O}(\delta)} +\frac{1}{4+\mathcal{O}(\delta)}+\frac{1}{\mathcal{O}(\delta)}\right|\to \infty,
	\end{split}
\end{align*}
% Thus, the right hand side of \eqref{inverse summation of third linearization} satisfies
% \begin{align*}
% 	\left| \frac{1}{-2\delta} +\frac{1}{4+\mathcal{O}(\delta)}+\frac{1}{-4+\mathcal{O}(\delta)}\right| \gg 1,
% \end{align*}
when $\delta\to 0$. Thus 
\begin{equation}\label{eq:third_lin_coef_nonzero}
\frac{1}{\mathbf{C}_{12}}+\frac{1}{\mathbf{C}_{13}}+	\frac{1}{\mathbf{C}_{13}}\neq 0 ,
\end{equation}
for all small enough $\delta>0$.

\begin{rmk}
 Let us define a Lorentz metric $\eta$ for $M$ by the formula 
 \[
  \eta(c_1e_1+V_1,c_2e_1+V_2):=\langle c_1e_1+\mathbf{i}V_1,c_2e_1+\mathbf{i}V_2\rangle,
 \]
where $c_1,c_2\in \R$ and $V_1,V_2\in TM_0$. We required that the vectors $\ol{\xi}_1,\ldots,\ol{\xi}_4$ in \eqref{eq:vecs_olxia} are lightlike vectors with respect to $\eta$ and sum up to $0$. The former requirement is because the corresponding phase functions need to satisfy the complex eikonal equation. The latter requirement is discussed in the next section.

% This is a requirement for the correspoding phase function satisfy the eikonal equation (on the geodesics). We will see that in order to cancel exponentially large factors and the non-stationary phase in our inverse problem, these vectors must sum up to $0$. 
A fact is that three $\eta$-lightlike vectors can only sum up to $0$, if the parts in $TM_0$ of two of them are linearly dependent. This would correspond to geodesics that do not intersect properly. Due to this geometric fact,  we overdifferentiate in this paper the nonlinearity $qu^2$ to obtain integral identities that consider more than three CGOs.
\end{rmk}

\subsection{Proof of $q_1^2=q_2^2$ (continued)}\label{sec:proof_third_lin_continued}
Let us then return to proving $q_1^2=q_2^2$. Let $\overline{\xi}_k$, $k=1,2,3,4$, be as in \eqref{eq:vecs_olxia}. We set
\[
 c_k=\abs{\xi_k} \text{ and } 
\]
and
\[
 s_1=c_1\tau+\mathbf{i}\lambda \quad \text{ and } \quad s_\ell=c_\ell\tau, \quad \text{ for } \quad  \ell=2,3,4.
\]
Then the corresponding CGOs are of the form
\begin{align*}
 v^{(1)}&= e^{(\abs{\xi_1}\tau+\mathbf{i}\lambda)  x_1} \left(\tau^{\frac{n-2}{8}}e^{\mathbf{i} (\abs{\xi_1}\tau+\mathbf{i}\lambda)\psi_1} a_1 + r_1 \right), \\
 v^{(2)}&= e^{-\abs{\xi_2}\tau x_1} \left(\tau^{\frac{n-2}{8}}e^{\mathbf{i}\abs{\xi_2}\tau \psi_2} a_2 + r_2 \right),\\
 v^{(3)}&= e^{\abs{\xi_3}\tau x_1} \left(\tau^{\frac{n-2}{8}}e^{\mathbf{i}\abs{\xi_3}\tau \psi_3} a_3 + r_3 \right),\\
 v^{(4)}&= e^{-\xi_4 \tau x_1} \left(\tau^{\frac{n-2}{8}}e^{\mathbf{i}\abs{\xi_4}\tau \psi_4} a_4 + r_4 \right).
\end{align*}

We may assume that $v^{(k)}$, $k = 1,\ldots, 4$ are supported in small enough neighborhoods of the corresponding geodesics $\gamma_k$ so that the mutual support of $v^{(k)}$ belongs to neighborhoods of the points where all the geodesics $\gamma_k$ intersect and where any pair of the geodesics intersect only once. Let us denote the points where all the geodesics $\gamma_k$ intersect by $p_0,p_1,\ldots, p_Q$.

Let $i\neq j\in \{1,2,3,4\}$, $i\neq j$ and $\beta=1,2$. By assumption, the DN maps of the equation \eqref{pf} for the potentials $q_1$ and $q_2$ satisfy $\Lambda_{q_1}=\Lambda_{q_2}$. By Proposition \ref{2nd_lin_sol_fixed_bndr} there are boundary values $f_{ij}$, which are the same for both $q_1$ and $q_2$, such that the solutions of the second linearized equations \eqref{equ_integral second id0} are of the form 
 \[
 w_\beta^{(ij)}=w_{0,\beta}^{(ij)}+ e^{\tau \Psi^{(ij)}}R_\beta^{(ij)},
 \]
 where the ingredients are as follows:
 \begin{align}\label{Psi kl_3rd}
 \begin{split}
	\Psi^{(ij)}&={((-1)^{1+i} c_i+ (-1)^{1+j}c_j) x_1+\mathbf{i}(c_i \psi_i+c_j \psi_j)}, \\
 w_{0,\beta}^{(ij)}&=\tau^{\frac{n-2}{4}}e^{((-1)^{1+i} s_i+ (-1)^{1+j}s_j) x_1+\mathbf{i}(s_i \psi_i+s_j \psi_j)}b^{(ij)}_\beta, \\
b^{(ij)}_\beta&=\tau^{-2}b^{(ij)}_{-2,\beta}+\cdots+ \tau^{-2N'}b^{(ij)}_{-2N',\beta}\\
 b_{-2,\beta}^{(ij)}&=\frac{2q_\beta}{((-1)^{1+i} c_i+(-1)^{1+j} c_j)^2-\abs{c_i\nabla'\psi_i+c_j\nabla'\psi_j}^2}a_0^{(i)}a_0^{(j)}.
 \end{split}
\end{align}
By \eqref{eq:nabla_on_geo}, at points of the form $(x_1,p_0)\in M$ we have  
\[
 b_{-2,\beta}^{(ij)}=\frac{2q_\beta}{\mathbf{C}_{ij}}a_0^{(i)}a_0^{(j)}.
\]
Here $a_0^{(i)}$ and $a_0^{(j)}$ are independent of the variable $x_1\in \R$. 
%The boundary value $f_{ik}$ is the same for both potentials $q_1$ and $q_2$.

To ease the following calculations, let us denote
\[
 \Psi_{1234}=\sum_{k=1}^4\left((-1)^{1+k}c_k x_1+\mathbf{i}c_k\s \psi_k\right)= \mathbf{i}\sum_{k=1}^4 c_k \psi_k
\]
and 
\[
 \Lambda_{1234}=\lambda_1(\mathbf{i}\s x_1-\psi_1).
\]
% where
% \begin{align}\label{Lambda ikl}
% 	\Lambda^{(ikl)}=\Lambda^{(ikl)}(x_1)=\left(\lambda_i + \lambda_k + \lambda_l \right) x_1, 
% \end{align}
% and
%Via the definition \eqref{Psi kl} of $\Psi^{(ik)}$, 
We see that 
\begin{align}\label{eq:exponential_large_cancel_proof}
	\begin{split}
	\Psi^{(12)}+\Psi^{(34)}
	=\Psi^{(13)}+\Psi^{(24)}
	=\Psi^{(14)}+\Psi^{(23)} 
	= \mathbf{i}\sum_{k=1}^4 c_k \psi_k=\Psi_{1234}.
	\end{split}
\end{align}
Since $\Psi_{1234}$ is purely imaginary at the intersection points $p_b$, $b=0,\ldots, Q$, the exponentially large linear factors will cancel in terms of the form $v^{(i)}w_\beta^{(jk)}v^{(l)}$ appearing the integral identity for the third linearization \eqref{equ_integral third id0}. %, for different $i,j,k,l=1,2,3,4$ and $\beta =1,2$.
We also have at the intersection point $p_0$ of the geodesics that
 \begin{align}\label{eq:oscillation_cancel_proof}
	\begin{split}
	\left.\nabla(\Psi^{(12)}+\Psi^{(34)})\right|_{p_0}
	&=\left. \nabla(\Psi^{(13)}+\Psi^{(24)})\right|_{p_0}
	=\left.\nabla(\Psi^{(14)}+\Psi^{(23)})\right|_{p_0} \\
 	&=\left. \mathbf{i}\sum_{k=1}^4 c_k \nabla\psi_k\right|_{p_0}=\left.\mathbf{i}\nabla\Psi_{1234}\right|_{p_0}=\mathbf{i}\sum_{k=1}^4 \xi_k=0.
	\end{split}
\end{align}
This implies that $p_0$ is a critical point of the phase functions of functions of the form $v^{(i)}w_\beta^{(jk)}v^{(l)}$. The critical point is also nondegenerate by \eqref{phase_pr} in Section~\ref{sec:CGOs} and thus we will be able to apply stationary phase in the asymptotic parameter $\tau$.

 Let us first consider the case $p_0$ is the only point where all the geodesics $\gamma_1,\ldots,\gamma_4$ intersect. With the above preparations and using $\Lambda_{q_1}=\Lambda_{q_2}$ the integral identity \eqref{integral id for 3rd linearization} of the third order linearization reads
\begin{align}\label{eq:integral_id_final_3rd}
%begin{split}
			0=& \int_{M} \Bigg[ q_1\left( v^{(i)} w_1^{(kl)} + v^{(k)}w_1^{(il)} + v^{(l)} w_1^{(ik)}\right)   \\
			       & \qquad\qquad\qquad\qquad\qquad  - q_2\left( v^{(i)} w_2^{(kl)} + v^{(k)}w_2^{(il)} + v^{(l)} w_2^{(ik)}\right)  \Bigg] v^{(m)}\, dV \nonumber \\
			       =&\tau^{-2}\tau^{\frac{n-2}{2}}\int_{M}e^{\tau\Psi_{1234}} \left[ e^{\Lambda_{1234}}  a_0^{(1)}a_0^{(2)}a_0^{(3)}a_0^{(4)}(q_1^2-q_2^2)\left( \mathbf{C}_{12}^{-1}+	\mathbf{C}_{13}^{-1}+\mathbf{C}_{23}^{-1}\right)  + \overline{R}\right]dV, \nonumber
%\end{split}
\end{align}
where $\overline{R}=\mathcal{O}_{L^1(M)}(\tau^{-1})$. The factor $\tau^{-2}$ arises from the amplitude functions of the solutions $w_\beta^{(ik)}$, see \eqref{Psi kl_3rd} and the power $\frac{n-2}{2}$ of $\tau$ is the sum of $\frac{3(n-2)}{8}$ and $\frac{n-2}{8}$. By \eqref{eq:exponential_large_cancel_proof} the exponentially large factors of the integrand cancel. Recall that the dimension of $M_0$ is $n-1$. 

We multiply the integral identity \eqref{eq:integral_id_final_3rd} by $\tau^{1/2}$ and $\tau^2$. This achieves the correct normalization $\tau^{\dim(M_0)/2}$ for stationary phase.  By \eqref{eq:oscillation_cancel_proof}, at the intersection point $p_0$ of the geodesics $\gamma_k$ for $x_1\in I\subset \R$ holds
\[
 \nabla \Psi_{1234}(x_1,p_0)=0.
\]
In normal coordinates $(y^1,\ldots,y^{n-1})$ centered at the point $p_0$ in $M_0$
\[ 
 \textrm{Re}\,\Psi_{1234}(y)= \sum_{j,k=1}^{n-1}A_{jk}y^jy^k+\mathcal{O}(|y|^3),
\]
for some negative definite matrix $A$ by the properties \eqref{phase_pr} of the phase functions. Note also that
\[
\tau^{\frac{n-1}{2}}\int_{\R^{n-1}}e^{-\tau |y|^2}\,dy = \mathcal{O}(1) \quad \text{ and } \quad \tau^{\frac{n-1}{2}}\int_{\R^{n-1}} |y|e^{-\tau |y|^2}\,dy = \mathcal{O}(\tau^{-\frac{1}{2}}).
\]
Thus, stationary phase shows that the limit $\tau\to\infty$ of \eqref{eq:integral_id_final_3rd} equals
\begin{align*}
	&\left.c_A\left(a_0^{(1)}a_0^{(2)}a_0^{(3)}a_0^{(4)}\right)\right|_{p_0}\left( \mathbf{C}_{12}^{-1}+	\mathbf{C}_{13}^{-1}+\mathbf{C}_{23}^{-1}\right)\\
	& \qquad \times \int_{\R}e^{\Lambda_{1234}(x_1,p_0)}(q_1^2(x_1,p_0)-q_2^2(x_1,p_0))dx_1,
\end{align*}
where 
\[
 c_A=\int_{\R^{n-1}}e^{x\cdot A x}dx\neq 0. 
\]
We refer to \cite[Proof of Theorem 5.1, Step 4]{LLLS2019inverse} for more details on this stationary phase argument. Here we have also used that $a_0^{(k)}$, $k=1,\ldots,4$, depend only on the transversal variables. We also continued $q_1$ and $q_2$ by zero from $I$ to $\R$ in the $x_1$ variable.

The geodesics $\gamma_k$ were parametrized so that $\gamma_k(0)=p_0$. Thus $\psi_k(p_0)=0$ and we have
\[
 e^{\Lambda_{1234}(x_1,p_0)}=e^{\mathbf{i}\lambda x_1}.
\]
Since $\mathbf{C}_{12}^{-1}+\mathbf{C}_{13}^{-1}+\mathbf{C}_{23}^{-1}\neq 0$ and $a_0^{(k)}|_{\gamma_k}\neq 0$ by \eqref{eq:third_lin_coef_nonzero} and \eqref{ampl_0_pr} respectively, combining the above shows that
\begin{align*}
% 	&\left.\left(a_0^{(1)}a_0^{(2)}a_0^{(3)}a_0^{(4)}\right)\right|_{p_0}\left( \mathbf{C}_{12}^{-1}+	\mathbf{C}_{13}^{-1}+\mathbf{C}_{23}^{-1}\right)\\
	&\int_{\R}e^{\lambda x_1}\left(q_1^2(x_1,p_0)-q_2^2(x_1,p_0)\right)dx_1=0.
\end{align*}
Inverting the Fourier transformation in the $x_1$ variable shows that  $q_1^2(x_1,p_0)=q_2^2(x_1,p_0)$. Since $p_0$ was arbitrary, this completes the proof in the case $p_0$ was the only point where all the geodesics intersect.

Consider then the remaining case where are several points $p_b$, $b=0,\ldots,Q$, where all all the geodesics $\gamma_1,\gamma_2, \gamma_3,\gamma_4$ intersect. 
Note also that outside (disjoint) neighborhoods $U_b$ of $p_b$ the function $e^{\tau\Psi_{1234}}$ is exponentially small. Thus, for different $i,j,k,l=1,2,3,4$, by normalizing and taking limit $\tau \to \infty$ of \eqref{eq:integral_id_final_3rd} we obtain
\begin{align}\label{eq:integral_id_final_3rd1}
\begin{split}
			&0=\lim_{\tau\to \infty} \sum_{b=0}^Q\int_{U_b} \Bigg[ q_1\left( v^{(i)} w_1^{(jk)} + v^{(j)}w_1^{(ik)} + v^{(k)} w_1^{(ij)}\right)  \\
			       &\qquad \qquad \quad \qquad \quad - q_2\left( v^{(i)} w_2^{(jk)} + v^{(j)}w_2^{(ik)} + v^{(k)} w_2^{(ij)}\right)  \Bigg] v^{(l)}\, dV  \\
			       &=\lim_{\tau\to\infty}\tau^{\frac{n-1}{2}} \sum_{b=0}^Q\int_{U_b} \Big[e^{\Lambda_{1234}} e^{\tau\Psi_{1234}} a_0^{(1)}a_0^{(2)}a_0^{(3)}a_0^{(4)}(q_1^2-q_2^2) \\
			       &\qquad\qquad\qquad\qquad\qquad \times \left( \mathbf{C}_{12}^{-1}(p_b)+	\mathbf{C}_{13}^{-1}(p_b)+\mathbf{C}_{23}^{-1}(p_b)\right)  \Big]dV. 
			       \end{split}
\end{align}
Here we have denoted
 \[
  \mathbf{C}_{ik}(p_b)=\left. \left\langle \nabla\left((-1)^{1+i}c_ix_1+\mathbf{i}c_i\s \psi_i\right), \nabla\left((-1)^{1+k}c_kx_1+\mathbf{i}c_k\s \psi_k\right)\right\rangle\right|_{p_b}\neq 0
 \]
 Note that $\mathbf{C}_{ik}(p_b)\neq 0$, since $\gamma_i$ and $\gamma_k$, $i\neq k$, intersect properly, cf. \eqref{eq:Cik_not_zero}.
%since $\gamma_i$ and $\gamma_j$ intersect properly, cf. \eqref{eq:Cik_not_zero}).
Therefore, by applying stationary phase to \eqref{eq:integral_id_final_3rd1} it follows that
 \[
  \sum_{b=0}^Q \hat{h}_b(\lambda)e^{c_b\lambda}=0, \quad \lambda\in \R.
 \]
 Here $c_b$ are the distinct geodesic parameter times of $\gamma_1$ where $\gamma_1(c_b)=p_b$ and
\begin{align*}
	  \hat{h}_b(\lambda):=\mathcal{F}_{x_1\to\lambda}\left(\left.a_0^{(1)}\cdots a_0^{(4)}\left( \mathbf{C}_{12}^{-1}+	\mathbf{C}_{13}^{-1}+\mathbf{C}_{23}^{-1}\right)\right|_{p_b}\left( q_1^2(x_1,p_b)-q_2^2(x_1,p_b)\right)\right),
\end{align*}
 where $\mathcal{F}_{x_1\to\lambda}$ is the Fourier transform in $x_1$ variable.
%  , and
%  \[
%   \mathbf{C}_{ik}|_{p_b}:=\left\langle \nabla\left((-1)^{1+i}c_ix_1+\mathbf{i}c_i\s \psi_i\right), \nabla\left((-1)^{1+k}c_kx_1+\mathbf{i}c_k\s \psi_k\right)\right\rangle|_{p_b}.
%  \]
% (Here $\mathbf{C}_{ik}|_{p_b}\neq 0$, since $\gamma_i$ and $\gamma_k$ intersect properly, cf. \eqref{eq:Cik_not_zero}.) 
By \cite[Lemma 6.2]{LLLS2019inverse} 
\[
h_0=\cdots=h_Q=0.
\]
Especially $q_1^2(x_1,p_0)=q_2^2(x_1,p_0)$, which concludes the proof of $q_1^2=q_2^2$ also in the case where there are several points where the geodesics $\gamma_k$ all intersect.
% 
% 
% 
% By \cite[Lemma 3.1]{salo2017normal}, we may find a nontangential geodesic for almost all points of $M_0$.\f{\textcolor{cyan}{Let's edit this paragraph if we assume that for $p_0\in M_0$ there is nontangential geodesic that does not have self-intersections.}}, repeating the above argument for almost all points and using continuity of $q_1$ and $q_2$ shows that
% \[
%  q_1^2=q_2^2.
% \]
% This completes the proof of $q_1^2=q_2^2$ in this simplified setting.

\subsection{Proof of $q_1=q_2$ and fourth order linearization}
We proved $q_1^2=q_2^2$ using third order linearizations of the equation $(-\Delta +V)u+qu^2=0$.  Here we consider fourth order linearizations of the equation and use it to complete the proof of Theorem \ref{t1}. Most of the steps here will be similar to those we used to prove $q_1^2=q_2^2$. However, the steps are somewhat more complicated.

Let $p_0\in M_0$, and let $\gamma_1$ be a non-tangential geodesic, which has no self-intersections. Let $\xi_1\in S_{p_0}M_0$ by the initial data if $\gamma_1$. Let us consider the equation
\begin{align}\label{equ_integral id0_4thlin}
	\begin{cases}
		(-\Delta +V)u_\beta +q_\beta u_\beta^2=0 & \text{ in }M,\\
		u_\beta =f &\text{ on }\p M,
	\end{cases}
\end{align}
this time with boundary values $f\in C^\infty(\p M)$ of the form \eqref{f_epsilon}, for $\beta=1,2$.
%\begin{align}%\label{f_epsilon}
	%f_\eps:=\displaystyle \sum_{i=1}^4\eps_i f_i+\sum_{i, j=1}^4\eps_i\eps_j f_{i j}+\sum_{i, j, k=1}^4\eps_i\eps_j\eps_k f_{ijk}.
%\end{align}
The first and second linearized equations are the same as before and read
\begin{align*}
 (-\Delta +V)v^{(i)}&=0, \\
 (-\Delta +V)w_\beta^{(ij)}&=-2q_\beta  v^{(i)}v^{(j)}.
\end{align*}
Here $v^{(i)}$ and $w_\beta^{(ij)}$, $i\neq j\in \{1,\ldots,4\}$, $\beta=1,2$, have boundary values $f_i$ and $f_{ij}$ respectively. The solutions $v^{(i)}$ are the same for both potentials $q_1$ and $q_2$. The third order linearizations $w_\beta^{(ijk)}$ now have (possibly) non-zero boundary values and satisfy 
\begin{align*}
	\begin{cases}
		(-\Delta +V)w_\beta^{(ijk)} =-2q_\beta\left( v^{(i)} w_\beta^{(jk)} + v^{(j)}w_\beta^{(ik)} + v^{(k)} w_\beta^{(ij)}\right)  & \text{ in }M,\\
		w_\beta^{(ijk)} =f_{ijk} &\text{ on }\p M,
	\end{cases}
\end{align*}
where $w_\beta^{(ijk)}=\left. \p^3 _{\eps_i \eps_j\eps_k}\right|_{\eps=0} u_\beta$,
for $\beta=1,2$ and different $i,j,k\in \{1,\ldots,4\}$. The boundary values $f_{ijk}$ are the same for both of the equations  \eqref{equ_integral third id0}, which correspond to the potentials $q_1$ and $q_2$. 

The fourth order linearization 
\[
 w^{(1234)}_\beta=\left. \p^4 _{\eps_1 \eps_2\eps_3\eps_4}\right|_{\eps=0} u_\beta
\]
is the solution of 
\begin{align}
	\begin{cases}
	(-\Delta+V)w_\beta^{(1234)}
		=-2q \left(  v^{(1)}w_\beta^{(234)} + v^{(2)} w_\beta^{(134)}   \right. \\		 
		\qquad \qquad \qquad \qquad  \qquad \qquad + v^{(3)} w_\beta^{(124)}+ v^{(4)}w_\beta^{(123)} \\
		\qquad \qquad \qquad \qquad \quad \qquad \qquad \left.  + w_\beta^{(12)}w_\beta^{(34)}+   w_\beta^{(13)}w_\beta^{(24)} + w_\beta^{(14)} w_\beta^{(23)}  \right) 
		&\text{ in }M,
		\\
		w_\beta^{(1234)}= 0 
		&\text{ on } \p M.
	\end{cases}
\end{align}
%for $\beta=1,2$.
Using $\Lambda_{q_1}=\Lambda_{q_2}$ we have by Lemma \ref{Lem:Integral identities} the integral identity
\begin{align}\label{fourth integral id_proof_sec_sec4}
\begin{split}
% 	\int_{\p M}& \left.\p^4 _{\eps_1 \eps_2 \eps_3\eps_4}\right|_{\epsilon=0} \left(   \Lambda_{q_1} f_\epsilon- \Lambda_{q_2} f_\epsilon \right) f_5\, dS \\
				 0&=\int_M \Bigg\{ q_1 \left(  v^{(1)}w_1^{(234)} + v^{(2)} w_1^{(134)}  + v^{(3)} w_1^{(124)}+ v^{(4)}w_1^{(123)}  \right.  \\
				&\qquad\qquad \qquad \qquad \qquad  \qquad  \left.  + w_1^{(12)}w_1^{(34)}+   w_1^{(13)}w_1^{(24)} + w_1^{(14)} w_1^{(23)}  \right) \\
				& \qquad \qquad  - q_2 \left(  v^{(1)}w_2^{(234)} + v^{(2)} w_2^{(134)}    + v^{(3)} w_2^{(124)}+ v^{(4)}w_2^{(123)}\right. \\
				& \left.\qquad \qquad \qquad \qquad \qquad \qquad    + w _2^{(12)}w_2^{(34)}+   w_2^{(13)}w_2^{(24)} + w_2^{(14)} w_2^{(23)}  \right)  \Bigg\} v^{(5)}\, dV.
				\end{split}
    \end{align}
% 
% By using the integral identity \eqref{fourth integral id} and \eqref{equal DN with derivatives} again, one has 
% \begin{align}\label{fourth integral id_proof_sec_sec4}
% 			\begin{split}
% 			 & \int_M \Bigg\{ q_1 \left(  v^{(1)}w_1^{(234)} + v^{(2)} w_1^{(134)}  + v^{(3)} w_1^{(124)}+ v^{(4)}w_1^{(123)}  \right.  \\
% 				&\qquad \qquad  \left.  + w_1^{(12)}w_1^{(34)}+   w_1^{(13)}w_1^{(24)} + w_1^{(14)} w_1^{(23)}  \right) \\
% 				& \qquad - q_2 \left(  v^{(1)}w_2^{(234)} + v^{(2)} w_2^{(134)}    + v^{(3)} w_2^{(124)}+ v^{(4)}w_2^{(123)}\right. \\
% 				& \left.\qquad \qquad \quad    + w _2^{(12)}w_2^{(34)}+   w_2^{(13)}w_2^{(24)} + w_2^{(14)} w_2^{(23)}  \right)  \Bigg\}v^{(5)} \, dV=0.
% 			\end{split}
% 		\end{align}

We will use five CGOs as the solutions $v^{(k)}$, $k=1,\ldots,5$. As before, these have the form %Thus $v^{(k)}$ are again on the form
\[
 v^{(k)} = e^{\pm s_k x_1}\left( \tau^{\frac{n-2}{8}} e^{\mathbf{i}s_k \psi_k} a^{(k)}_\tau + r_\tau^{(k)}\right),
\]
where $s_k=c_k\tau+\mathbf{i}\lambda_k$. However, the geodesics of $(M_0,g_0)$ corresponding to the phase functions $\psi_k$ will be different from what we used earlier. We choose the geodesics so that each pair of different ones of them intersect properly. This is as before. However, we additionally require the geodesics to be so that
\[
 (\pm c_i\pm c_j\pm c_k)^2-\abs{c_i\dot \gamma_i+c_j\dot \gamma_j+c_k\dot \gamma_k}^2\neq 0,
\]
when all the geodesics $\gamma_k$ intersect. This is the condition $\langle \nabla\widetilde \Psi, \nabla\widetilde \Psi\rangle\neq 0$ of Lemma \ref{3rd_lin_sol_free_bndr} and Proposition \ref{3rd_lin_sol_fixed_bndr}. 

% 
%  We choose $v^{(k)}$, $k=1,2,3,4$, to be CGOs corresponding to pairs of properly intersecting geodesics on $(M_0,g_0)$. 
With suitable choices of other geodesics $\gamma_2,\gamma_3, \gamma_4, \gamma_5$, and coefficients $c_k$, $k=1,\ldots,5$, we show that the integrand on the right hand side of \eqref{fourth integral id_proof_sec_sec4} restricted to a neighborhood of $p_0$ in $M_0$ is close to a multiple of the delta function at $p_0$. 
%  
 %Let $p_0\in M_0$, and let $\gamma_1$ be a non-tangential geodesic with initial data $\xi_1\in S_{p_0}M_0$. 
% We choose the other geodesics and coefficients $c_k$ next. 

\subsection{Choices of vectors for the fourth order linearization}\label{sec:choice_of_vecs_4th_lin}
The fourth order linearization $w^{(1234)}$ of $(-\Delta_q+V)+qu^2=0$ satisfies
\begin{align}\label{w_eq_4th_choosing_vectors_motivation}
	\begin{split}
	(-\Delta_{g}+V)w^{(1234)} 	=-&2q \left(  v^{(1)}w^{(234)} + v^{(2)} w^{(134)}  + v^{(3)} w^{(124)}+ v^{(4)}w^{(123)}  \right.\\
		&\qquad  \left. + w ^{(12)}w^{(34)}+   w^{(13)}w^{(24)} + w^{(14)} w^{(23)}  \right)  \quad \text{ in }M.
	\end{split}
\end{align}
Our aim is to show that the solution $w^{(1234)}$ behaves like $v^{(1)}v^{(2)}v^{(3)}v^{(4)}$ up to a multiplication by an amplitude function for $\tau$ sufficiently large.

\medskip

\noindent $\bullet$ \textbf{Failed choices of vectors.} Let us first discuss why the earlier vectors $\overline \xi_i$, $i=1,2,3,4$, do not work here. If we use the earlier vectors \eqref{eq:vecs_olxia} and the corresponding CGOs we will find that for example $w^{(123)}$ in \eqref{w_eq_4th_choosing_vectors_motivation} solves
\begin{align}\label{eq:eqn_type_motivation}
\begin{split}
		(-\Delta_{g}+V)w^{(123 )}&=-2q\left( v^{(1)} w^{(23)} + v^{(2)}w^{(13)} + v^{(3)} w^{(12)}\right) \\
		&= e^{\tau \sum_{j\in \{1,2,3\}}\abs{\xi_j}((-1)^{j+1}x_1+\mathbf{i}\psi_j)}\check a ,
		\end{split}
\end{align}
where $\check a$ is some amplitude function whose precise form is not important for this discussion. At the intersection points of the geodesics corresponding to $\xi_a$
\[
 \nabla\left(\sum_{j=1}^3\abs{\xi_j}((-1)^{j+1}x_1+\mathbf{i}\psi_j)\right)=\sum_{j=1}^3\overline \xi_j=-\overline \xi_4
\]
by \eqref{eq:lightlike_sums_to_zero}. Now, if we try a WKB ansatz of the form $e^{\tau \sum_{a\in \{2,3,4\}}\abs{\xi_a}((-1)^{a+1}x_1+\mathbf{i}\psi_a)}\check b$ to solve \eqref{eq:eqn_type_motivation}, where $\check b$ is an amplitude function, we end up dividing by $\langle \overline \xi_4, \overline \xi_4\rangle$, which is $0$.  
%This however poses a problem since $\overline \xi_1\cdot \overline \xi_1=0$. 
Consequently, the ansatz does not work  and we need to use vectors that are different than $\overline \xi_a$.
%we would be dividing by zero in the construction we used earlier to solve equation of the type \eqref{eq:eqn_type}.

\medskip

\noindent $\bullet$ \textbf{Successful choices of vectors.} We choose new vectors to define the CGOs $v^{(k)}$, such that we can apply these CGOs to achieve our target. Denote the vectors by 
\[
\overline{\zeta}_k, \quad k=1,2,3,4,5. 
\]
Let $\delta>0$ and let $\xi_j$, $j=1,2,3,4$, be as in Section \ref{sec:choice_of_vecs_3rd_lin}. Especially $\langle \xi_1,\xi_2\rangle =1-\delta$ and $\abs{\xi_1}=\abs{\xi_2}=1$. Note that the integral identity \eqref{fourth integral id_proof_sec_sec4} implicitly concerns $5$ possibly different $v^{(k)}$. We choose the vectors $\zeta_k\in T_{p_0}M_0$ as follows
\begin{align*}
 &\zeta_1=\xi_1, \qquad \zeta_2=\xi_2, \qquad  \\
 &\zeta_3=
 \left(1+\sqrt{\frac{2}{2-\delta}}\right)\xi_3 ,  \quad \zeta_4=\left(1+\sqrt{\frac{2}{2-\delta}}\right)\xi_4, \\
\text{ and } \quad &\zeta_5=\sqrt{\frac{2}{2-\delta}}(\xi_1+\xi_2).
\end{align*}
Note that $\abs{\zeta_5}=2$. 
We define $\ol\zeta_k$ by 
\begin{align*}
& \ol \zeta_1=\abs{\zeta_1}e_1+ \mathbf{i}\zeta_1, \quad \ol \zeta_2=\abs{\zeta_2}e_1+\mathbf{i}\zeta_2, \\
 &\ol \zeta_3=\abs{\zeta_3}e_1+\mathbf{i}\zeta_3, \quad \ol \zeta_4=-\abs{\zeta_4}e_1+\mathbf{i}\zeta_4,\\
  \text{ and }\quad & \ol \zeta_5=-\abs{\zeta_5}e_1+\mathbf{i}\zeta_5.
\end{align*}
 We also define 
\[
 c_k=\abs{\zeta_k}.
\]
In particular, we have $c_1=c_2=1$ and $c_5=2$. Then 
\begin{align}\label{eq:sum_zeta_zero}
	\sum_{j=1}^5 \zeta_j=&\xi_1+\xi_2-\left(1+\sqrt{\frac{2}{2-\delta}}\right)\left(\frac{1}{1+\delta}\xi_1+\frac{\delta}{1+\delta} \xi_2+\frac{\delta}{1+\delta}\xi_1+\frac{1}{1+\delta} \xi_2\right)\nonumber\\
	&+ \sqrt{\frac{2}{2-\delta}}(\xi_1+\xi_2) =0.
\end{align}
and
\begin{equation}\label{eq:sum_real_part_olzeta_zero}
\text{Re}\left(\sum_{j=1}^5 \ol \zeta_j\right)=0.%+e_1-\sqrt{(2(2-\delta)}    
\end{equation}
Consequently, the sum of the vectors $\ol \zeta_a$ vanishes:
\begin{equation}\label{eq:sum_of_ol_zeta_4th}
 \sum_{j=1}^5 \ol \zeta_j=0.
\end{equation}
The condition \eqref{eq:sum_of_ol_zeta_4th} will imply that the non-stationary phase at $p_0$ and exponentially growing factors of in the integrand of the integral identity \eqref{fourth integral id_proof_sec_sec4} will cancel out. We showed in Section \ref{sec:choice_of_vecs_3rd_lin} that the vectors $\xi_1,\ldots,\xi_4$ are pairwise linearly independent. Consequently $\zeta_1,\ldots,\zeta_4$ are pairwise linearly independent. We also see that $\zeta_5$ is not proportional to any of the other vectors $\zeta_1,\ldots,\zeta_4$. It follows that the geodesics of $(M_0,g_0)$ corresponding to $\zeta_1,\ldots,\zeta_5$ intersect properly. Since the vectors $\zeta_2,\ldots, \zeta_5$ are up to scalings small perturbations of $\xi_1$, the corresponding geodesics are nontangential and they do not have self-interactions. 
%\f{Add Lemma reference, or where this fact is earlier stated. -Tony}

We then consider how solutions to \eqref{w_eq_4th_choosing_vectors_motivation}, which correspond to the CGOs determined by the vectors $\ol \zeta_j$, $j=1,2,3,4$, look like. 
Let us first note that at the intersection points of the corresponding geodesics
\begin{equation}\label{eq:condition_for_4th_ord_lemma}
 (\pm c_i\pm c_j\pm c_k)^2-\abs{c_i\nabla'\psi_i+c_j\nabla'\psi_j+c_k\nabla'\psi_k}^2\neq 0
\end{equation}
for all indices $i,j,k\in \{2,3,4\}$ which are all different. Indeed, by \eqref{eq:sum_of_ol_zeta_4th} we note that
\[
 (\pm c_i\pm c_j\pm c_k)^2-\abs{c_i\nabla'\psi_i+c_j\nabla'\psi_j+c_k\nabla'\psi_k}^2=\langle \ol\zeta_l+\ol\zeta_m, \ol\zeta_l+\ol\zeta_m\rangle,
\]
where $l,m \{1,2,3,4,5\}$ are the unique two different indices, which do not belong to the set $\{j,k,l\}$. Then, we have
\begin{align*}
	\langle \ol\zeta_l+\ol\zeta_m, \ol\zeta_l+\ol\zeta_m\rangle=2\langle \ol\zeta_l, \ol\zeta_m\rangle=2(\pm c_lc_m  -\langle \zeta_l, \zeta_m\rangle)\neq 0,
\end{align*}
since $\abs{\langle \zeta_l, \zeta_m\rangle}< \abs{\zeta_l}\abs{\zeta_m}=c_lc_m$. Here we have the strict inequality since $\zeta_l$ and $\zeta_m$ are linearly independent.

By \eqref{eq:condition_for_4th_ord_lemma}, we may apply Lemma \ref{3rd_lin_sol_free_bndr}. Thus, by having the restrictions of the supports of $v^{(k)}$ to $M_0$ in small enough neighborhoods of the corresponding geodesics, the solution $w^{(123)}$ to third linearization 
\begin{align*}%\label{equ_integral third id0}
	%\begin{cases}
		(-\Delta +V)w^{(123)} =-2q\left( v^{(1)} w_j^{(23)} + v^{(2)}w^{(13)} + v^{(3)} w_j^{(12)}\right)
%		w_j^{(ikl)} =f_{ikl} &\text{ on }\p M.
%
\end{align*}
up to a correction term is given by a WKB ansatz with amplitude of the form \eqref{expansion_formula_new_B}. The leading order coefficient of the ansatz is 
%\f{There is mismatch on how $B_{-4}$ is defined. Should have $q^2$ as below as a factor, not $q^3$ as somewhere else. Let's fix it. -Tony}
\begin{align*}
    B_{-4}^{(123)}=&4q^2a_0^{(1)}a_0^{(2)}a_0^{(3)}\frac{1}{(\pm c_1\pm c_2\pm c_3)^2-\abs{c_1\nabla'\psi_1+c_2\nabla'\psi_2+c_3\nabla'\psi_3}^2}\nonumber\\
    &\quad \times \Bigg(\frac{1}{(\pm c_1\pm c_2)^2-\abs{c_1\nabla'\psi_1+c_2\nabla'\psi_2}^2} \\
    & \quad \qquad +\frac{1}{(\pm c_1\pm c_3)^2-\abs{c_1\nabla'\psi_1+c_3\nabla'\psi_3}^2}\nonumber \\
    & \quad \qquad +\frac{1}{(\pm c_2\pm c_3)^2-\abs{c_2\nabla'\psi_2+c_3\nabla'\psi_3}^2}\Bigg). 
	\end{align*}
	Note the power $2$ of the potential $q$ in $B_{-4}^{(123)}$. Let us define for $i,j,k\in \{1,\ldots,4\}$ all different
\begin{align}\label{defi of coefficient D}
\begin{split}
	& \mathbf{D}_{ij}=\langle \ol \zeta_i+\ol \zeta_j, \ol \zeta_i+\ol \zeta_j\rangle, \\
	\text{ and }\quad &  \mathbf{D}_{ijk}=\langle\ol \zeta_i +\ol \zeta_j+\ol \zeta_k,\ol \zeta_i+\ol \zeta_j+\ol \zeta_k\rangle.
\end{split}
\end{align}
At an intersection point $p_0$ of the geodesics, we thus have is 
\[
 B_{-4}^{(123)}(p_0)=4q^2a_0^{(1)}a_0^{(2)}a_0^{(3)}\frac{1}{\mathbf{D}_{123}}\left(\frac{1}{\mathbf{D}_{12}+\mathbf{D}_{13}+\mathbf{D}_{23}}\right).
\]
We have similar formulas for the leading order coefficients of the ansatzes for $w^{(234)}$, $w^{(134)}$ and $w^{(124)}$. 
Furthermore, by \eqref{eq:sum_of_ol_zeta_4th} we have 
\begin{align}\label{D123=D45}
	\begin{split}
		\mathbf{D}_{123}&=\langle\ol \zeta_1 +\ol \zeta_2+\ol \zeta_3, \ol \zeta_1+\ol \zeta_2+\ol \zeta_3\rangle 
		= \langle \ol \zeta_4+ \ol \zeta_5, \ol \zeta_4+ \ol \zeta_5 \rangle =\mathbf{D}_{45}, \\
		\mathbf{D}_{234}&=\langle\ol \zeta_1+\ol \zeta_5, \ol \zeta_1+\ol \zeta_5\rangle =\mathbf{D}_{15},\\
		\mathbf{D}_{134}&=\mathbf{D}_{25}, \\
		\mathbf{D}_{124}&= \mathbf{D}_{35}.
	\end{split}
\end{align}
Therefore, by using \eqref{D123=D45}, the solution $w^{(1234)}$ to the fourth order linearization will be (up to a correction term) of the form
\begin{align}\label{w_eq_4th_choosing_vectors_formal}
\begin{split}
&\left[\frac{1}{\mathbf{D}_{15}}\left(\frac{1}{\mathbf{D}_{23}+\mathbf{D}_{24}+\mathbf{D}_{34}}\right) + \frac{1}{\mathbf{D}_{25}}\left(\frac{1}{\mathbf{D}_{13}+\mathbf{D}_{14}+\mathbf{D}_{34}}\right)   \right. \\		 
		&\qquad  + \frac{1}{\mathbf{D}_{35}}\left(\frac{1}{\mathbf{D}_{12}+\mathbf{D}_{14}+\mathbf{D}_{24}}\right) + \frac{1}{\mathbf{D}_{45}}\left(\frac{1}{\mathbf{D}_{12}+\mathbf{D}_{13}+\mathbf{D}_{23}}\right) \\
		&\qquad   \left. \left.  + \frac{1}{\mathbf{D}_{12}}\frac{1}{\mathbf{D}_{34}}+  \frac{1}{\mathbf{D}_{13}}\frac{1}{\mathbf{D}_{24}} + \frac{1}{\mathbf{D}_{14}}\frac{1}{\mathbf{D}_{23}}\right)\right] \\
		& \qquad \quad \times e^{\tau \sum_{j\in \{1,2,3,4\}}(\text{Re}(\zeta_j)x_1+\mathbf{i}\abs{\zeta_j}\psi_j}\widetilde A.
		\end{split}
\end{align}
Here $\widetilde{A}$ is an amplitude function which has (up to a multiplication by a power of $\tau$) the leading order coefficient 
%\f{Here the power $q^3$ is correct. -Tony} \f{Could you Ali and Yi-Hsuan please check, and comment if needs more details.} 
\[
8q^3a_0^{(1)}a_0^{(2)}a_0^{(3)}a_0^{(4)}.
\]
Note the power $3$ of the potential $q$ here.

Similar to Section \ref{sec:choice_of_vecs_3rd_lin}, where we showed that the factor \eqref{eq:third_lin_coef_nonzero} of the third order linearization is non-zero, we may show that the coefficient in the brackets of \eqref{w_eq_4th_choosing_vectors_formal}, call it $\mathbf{E}_\delta$ is not zero. We have:

\begin{lem}\label{lem:Edelta}
	The quantity 
	\begin{align}
		\begin{split}
			\mathbf{E}_\delta &= \frac{1}{\mathbf{D}_{15}}\left(\frac{1}{\mathbf{D}_{23}+\mathbf{D}_{24}+\mathbf{D}_{34}}\right) + \frac{1}{\mathbf{D}_{25}}\left(\frac{1}{\mathbf{D}_{13}+\mathbf{D}_{14}+\mathbf{D}_{34}}\right)    \\		 
			&\qquad  + \frac{1}{\mathbf{D}_{35}}\left(\frac{1}{\mathbf{D}_{12}+\mathbf{D}_{14}+\mathbf{D}_{24}}\right) + \frac{1}{\mathbf{D}_{45}}\left(\frac{1}{\mathbf{D}_{12}+\mathbf{D}_{13}+\mathbf{D}_{23}}\right) \\
			&\qquad     + \frac{1}{\mathbf{D}_{12}}\frac{1}{\mathbf{D}_{34}}+  \frac{1}{\mathbf{D}_{13}}\frac{1}{\mathbf{D}_{24}} + \frac{1}{\mathbf{D}_{14}}\frac{1}{\mathbf{D}_{23}} = \mathcal{O}(\delta^{-3}) \neq 0,
		\end{split}
	\end{align}
for all sufficiently small $\delta>0$.
\end{lem}
The proof of the lemma is elementary, but involves rather long calculations. We have placed the proof in Appendix \ref{Appendix D_ikl}.

\subsection{Proof of $q_1=q_2$ (continued)}
Let us then return to proving $q_1=q_2$. Let $\overline{\zeta}_k$, $k=1,\ldots,5$, be as in Section \ref{sec:choice_of_vecs_4th_lin} above. We have
\[
 c_k=\abs{\zeta_k} 
\]
and we set
\[
 s_1=c_1\tau+\mathbf{i}\lambda, \text{ and } s_k=c_k, \quad k=2,3,4,5.
\]
The CGOs corresponding to vectors $\ol \zeta_5$ are of the form
\begin{align}\label{eq:choices_for_CGOs_4th_lin}
\begin{split}
 v^{(1)}&= e^{(\abs{\zeta_1}\tau+\mathbf{i}\lambda)  x_1} \left(\tau^{\frac{n-2}{8}}e^{\mathbf{i} (\abs{\zeta_1}\tau+\mathbf{i}\lambda)\psi_1} a_1 + r_1 \right), \\
 v^{(2)}&= e^{\abs{\zeta_2}\tau x_1} \left(\tau^{\frac{n-2}{8}}e^{\mathbf{i} \abs{\zeta_2}\tau \psi_2} a_2 + r_2 \right),\\
 v^{(3)}&= e^{\abs{\zeta_3}\tau x_1} \left(\tau^{\frac{n-2}{8}}e^{\mathbf{i} \abs{\zeta_3}\tau \psi_3} a_3 + r_3 \right),\\
 v^{(4)}&= e^{-\abs{\zeta_4} \tau x_1} \left(\tau^{\frac{n-2}{8}}e^{\mathbf{i} \abs{\zeta_4}\tau \psi_4} a_4 + r_4 \right),\\
 v^{(5)}&= e^{-\abs{\zeta_5}\tau x_1} \left(\tau^{\frac{n-2}{8}}e^{\mathbf{i}\abs{\zeta_5}\tau \psi_5} a_5 + r_5 \right).
 \end{split}
\end{align}

Since $\Lambda_{q_1}(f_\eps)=\Lambda_{q_2}(f_\eps)$, by Propositions \ref{2nd_lin_sol_fixed_bndr} and  \ref{3rd_lin_sol_fixed_bndr} there are boundary values $f_{ij}$ and $f_{ijk}$, $i,j,k=1,2,3,4$, such that the solutions of the second linearized equations \eqref{equ_integral second id0}
and third linearized equations \begin{align}%\label{equ_integral third id0}
\begin{cases}
		(-\Delta +V)\omega_\beta^{(ijk)} =-2q_\beta\left( v^{(i)} w_\beta^{(jk)} + v^{(j)}w_\beta^{(ik)} + v^{(k)} w_\beta^{(ij)}\right)  & \text{ in }M,\\
	\omega_\beta^{(ijk)} =f_{ijk} &\text{ on }\p M,
	\end{cases}
\end{align}
for $\beta=1,2,$ and $i,j,k$ all different, which are of the form
\begin{align*}
	 w_\beta^{(ij)}=&w_{0,\beta}^{(ij)}+ e^{\tau \Psi^{(ij)}}R_\beta^{(ij)},\\
	 \text{and }\quad \omega_\beta^{(ijk)}=&\omega_{0,\beta}^{(ijk)}+ e^{\tau \widetilde\Psi^{(ijk)}}\widetilde R_\beta^{(ijk)}.
\end{align*}
 For given $K,N\in \N\cup \{0\}$, the correction terms $R_\beta^{(ij)}$ and $\widetilde R_\beta^{(ijkl)}$ can be assumed to be $\mathcal{O}_{L^2(M)}(\tau^{-N})$ by taking the amplitude expansions of the CGOs $v^{(k)}$, $k=1,2,3,4$ to be precise enough (i.e. $N'$ large enough). We refer to Propositions \eqref{2nd_lin_sol_fixed_bndr} and \eqref{3rd_lin_sol_fixed_bndr} for the specifics of $w_\beta^{(ij)}$ and $w_\beta^{(ijk)}$. 
 
The phase functions $\Psi^{(ij)}$ and $\widetilde\Psi^{(ijk)}$ satisfy at the point $p_0$ where all the geodesics $\gamma_1,\ldots,\gamma_5$ intersect
 \begin{align}\label{Psi ikl}
 \begin{split}
	\Psi^{(ij)}(p_0)&=\ol \zeta_i +\ol \zeta_j, \\  
 	\widetilde\Psi^{(ijk)}(p_0)&=\ol \zeta_i +\ol \zeta_j+\ol\zeta_k.
 \end{split}
\end{align}
The leading order coefficients of the amplitudes of $w_\beta^{(ik)}$ and $\omega_\beta^{(ikl)}$ are
 \begin{align}\label{A4 ikl}
 \begin{split}
	b_{-2,\beta}^{(ik)}(p_0)&=\frac{2q_\beta}{\mathbf{D}_{ik}}a_0^{(i)}a_0^{(k)},\\ 
 	B_{-4,\beta}^{(ikl)}(p_0)&=4q_\beta^2a_0^{(i)}a_0^{(k)}a_0^{(l)}\frac{1}{\mathbf{D}_{ikl}}\left(\frac{1}{\mathbf{D}_{ik}+\mathbf{D}_{il}+\mathbf{D}_{kl}}\right).
 \end{split}
\end{align}

Let us denote $\Psi_{12345}$ as the sum of all the phase functions of $v^{(1)},\ldots, v^{(5)}$ in \eqref{eq:choices_for_CGOs_4th_lin}, where $\tau$ is a parameter. More precisely, $\Psi_{12345}$ is given as 
\[
 \Psi_{12345}=\Big(\abs{\zeta_1}+\abs{\zeta_2}+\abs{\zeta_3}-\abs{\zeta_4}-\abs{\zeta_5}\Big)x_1+\mathbf{i}\Big(\abs{\zeta_1}\psi_1+\abs{\zeta_2}\psi_2+\abs{\zeta_3}\psi_3+\abs{\zeta_4}\psi_4+\abs{\zeta_5}\psi_5\Big).
\]
Let us also set
\[
 \Lambda_{1234}=\lambda_1(\mathbf{i}\s x_1-\psi_1).
\]
At the point $p_0$ where all the geodesics $\gamma_1,\ldots,\gamma_5$ intersect
\begin{equation}\label{eq:oscillation_cancel_zeta_proof}
 \nabla\Psi_{12345}(x_1,p_0)=0
\end{equation}
for $x\in I\subset \R$ by \eqref{eq:sum_zeta_zero}, and
\begin{equation}\label{eq:exponential_large_cancel_zeta_proof}
 \mathrm{Re}(\Psi_{12345})(x_1,p_0)=0
\end{equation}
by \eqref{eq:sum_real_part_olzeta_zero}. The condition \eqref{eq:exponential_large_cancel_zeta_proof} implies that $\Psi_{12345}$ is not exponentially growing in $\tau$. Moreover, by \eqref{eq:oscillation_cancel_zeta_proof} we have that $p_0$ is a critical point of $\Psi_{12345}$. By the properties of $\psi_k$, the point $p_0$ is also nondegenerate, see \eqref{phase_pr}.

We multiply the right hand side of the integral identity  of the fourth order linearization \eqref{fourth integral id_proof_sec_sec4} by $\tau^{4}\tau^{1/2}$ and take the limit $\tau\to \infty$. In the case $p_0$ is the only point where all the geodesics $\gamma_1,\ldots,\gamma_5$ intersect, by stationary phase the limit tends to
\[
 0=c_{\tilde{A}}\mathbf{E}_\delta\left. \left(a_0^{(1)}a_0^{(2)}a_0^{(3)}a_0^{(4)}a_0^{(5)}\right)\right|_{p_0}\int_\R e^{i\lambda x_1}(q_1^3(x_1,p_0)-q_2^3(x_1,p_0))\, dx_1,
\]
where $\mathbf{E}_\delta$ is the coefficient of $w^{(1234)}$ in Lemma \ref{lem:Edelta} and $c_{\tilde{A}}\neq 0$ is given by a similar formula as $c_A$ in Section \ref{sec:proof_third_lin_continued}.  Here we also used that $a_0^{(1)},\ldots, a_0^{(5)}$ are independent of $x_1$. By Lemma \ref{lem:Edelta}, the coefficient $\mathbf{E}_\delta\neq 0$ for all small enough $\delta>0$.  
Inverting, the Fourier transformation in the variable $x_1$ shows that $q_1^3(x_1,p_0)=q_2^3(x_1,p_0)$. Thus 
\[
q_1(x_1,p_0)=q_2(x_1,p_0) 
\]
for $x_1\in \R$. 
If there were several points where $\gamma_1,\ldots,\gamma_5$ intersect, we argue similarly as in Section \ref{sec:proof_third_lin_continued} by using \cite[Lemma 6.2]{LLLS2019inverse}. Since $p_0$ was arbitrary, this completes the proof.

\appendix

\section{Boundary determination}\label{Section: Boundary determination}
We prove that the DN map of the semilinear elliptic equation 
\[
(-\Delta_g+V)u +qu^m=0 \text{ in } M,  \quad u=f \text{ on } \p M
\]
on a compact smooth Riemannian manifold with boundary determines the formal Taylor series (the jet) of the coefficient $q$ (in the boundary normal coordinates) on the boundary. Here, $m\geq 2$ is an integer, and $V$ and $q$ are smooth functions on $M$. We assume also that zero is not a Dirichlet eigenvalue for the operator $-\Delta_g+V$ on $M$.

We expect this result to be well-known to experts on the field, but could not find a reference on it, so we offer detailed characterization and its proof.

\begin{prop}[Boundary determination]\label{Prop_boundary_det}
	For $m\geq 2$, $m\in \N$, let $(M,g)$ be a compact Riemannian manifold with $C^\infty$ boundary $\p M$ and consider the boundary value problem % for emilinear elliptic operator 
	\begin{align}\label{eq:boundary_det_equation_app}
		\begin{cases}
			(-\Delta_g +V) u+qu^m=0 & \text{ in } M,\\
			u= f & \text{ on } \p M.
		\end{cases}
	\end{align}
	Assume that the DN map $\Lambda_q$ of the equation \eqref{eq:boundary_det_equation_app} is known for small boundary values. Then $\Lambda_q$ determines the formal Taylor series of $q$ on the boundary $\p M$. 
	
	In addition, if $f\in C^{2,\alpha}(\p M)$ is so small that \eqref{eq:boundary_det_equation_app} has a unique small solution, the DN map determines the formal Taylor series of the solution $u=u_f$ at any point on the boundary.
\end{prop}

% 
% \begin{prop}[Boundary determination]\label{Prop_boundary_det}
% 	For $m\geq 2$, $m\in \N$, let $(M,g)$ be a compact Riemannian manifold with $C^\infty$ boundary $\p M$ and consider the boundary value problem % for emilinear elliptic operator 
% 	\begin{align}\label{bnd_val_prob1}
% 		\begin{cases}
% 			(-\Delta_g +V) u+qu^m=0 & \text{ in } M,\\
% 			u= f & \text{ on } \p M.
% 		\end{cases}
% 	\end{align}
% 	Assume that the DN map $\Lambda_q$ of this equation known for small boundary values. Then $\Lambda_q$ determines the formal Taylor series of $q$ on the boundary $\p M$. In addition, the DN map $\Lambda_q$ determines the formal Taylor series of the solution $u$ at any point on the boundary.
% \end{prop}

\begin{proof}
	
	{\it Determination of Taylor expansion of $q$:}
	
	\medskip
	
	\noindent Let $f\in C^{2,\alpha}(\p M)$. We consider boundary values $f_0\in C^{2,\alpha}(\p M)$ and $f_t=f_0 +tf$ and assume that $\norm{f_0}_{C^{2,\alpha}(\p M)}$ and $|t|$ are sufficiently small so that the DN maps of $f_0$ and $f_t$ are both well-defined. We denote by $u_0$ and $u_t$, the unique solutions of \eqref{eq:boundary_det_equation_app} with boundary data $f_0$ and $f_t$ on $\p M$, respectively. By linearizing the equation \eqref{eq:boundary_det_equation_app} at $t=0$, we obtain  
	\begin{align}\label{bnd_val_prob2}
		\begin{cases}
			-\Delta_g v +\left(V+mqu_0^{m-1}\right)v=0 & \text{ in } M,\\
			v= f & \text{ on } \p M,
		\end{cases}
	\end{align}
	where $v=\displaystyle\lim_{t\to 0}\frac{u_t-u_0}{t}$ and $u_0$ solves
	\begin{align}\label{bnd_val_prob3}
		\begin{cases}
			(-\Delta_g +V) u_0+qu_0^m=0 & \text{ in } M,\\
			u_0= f_0 & \text{ on } \p M.
		\end{cases}
	\end{align}
	Moreover, $v$ is the solution of 
	\begin{align*}
		\begin{cases}
			-\Delta_g  v+ \widetilde{q}v=0 & \text{ in } M,\\
			v=f & \text{ on } \p M.
		\end{cases}
	\end{align*}
	where
	\begin{align}\label{tilde q}
		\widetilde{q}:=V+mqu_0^{m-1} \text{ in }M.
	\end{align}

	Since we know the DN map of the boundary value problem~\eqref{eq:boundary_det_equation_app}, we know the DN map of the linearized problem~\eqref{bnd_val_prob2}. This is proven in~\cite[Proposition 2.1]{LLLS2019inverse}, where it is shown that the DN map is $C^\infty$ in the Frech\'et sense.
	%and we know the DN map of the linearized equation~\eqref{bnd_val_prob2}. 
	It follows by~\cite[Theorem 8.4.]{ferreira2009limiting} that we know the formal Taylor series of $\widetilde{q}$ on $\p M$.  In particular, by choosing
	$$
	u_0 =f_0=\vareps_0>0 \text{ on } \p M,
	$$
	for some sufficiently small constant $\vareps_0>0$, and noting that 
	\[
	q=\frac{\widetilde{q}-V}{m\vareps_0^{m-1}}\text{ on }  \p M,
	\]
	it follows that we know the  of $q$ on the boundary $\p M$.
	%$u=\eps_0$ on $\p M$ and since we know $\tilde{q}$ on $\p M$, it follows that we know $q$ on $\p M$. 
	
	Next we determine first order derivatives of $q$ on the boundary. Given a point $x_0 \in \p M$, let $x=(x_1,\ldots, x_n)\in \p M$ be boundary normal coordinates near $x=x_0$ in $M$. 
	Differentiating \eqref{tilde q} yields
	\begin{align}\label{first order of q at boundary}
		\begin{split}
			\p_{x_n}\widetilde{q}=&m\p_{x_n}(u_0^{m-1})q+m(\p_{x_n}q)u_0^{m-1} +\p_{x_n} V\\
			= &m(m-1)u_0^{m-2}( \p_{x_n}u_0 ) q +m(\p_{x_n}q)u_0^{m-1}+\p_{x_n} V.
		\end{split}
	\end{align}
% 	for $k=1,2,\ldots, n$, where $u_0|_{\p M}=f_0$ can be arbitrary such that $\norm{f_0}_{C^{2,\alpha}(\p M)}$ is sufficiently small. 
% 	
Since we have already determined the Taylor series of $\tilde q$ on the boundary and
\[
	\p_{x_n}u_0=\Lambda_q(f_0),
\] 
we may determine $\p_{x_n}q$ by solving it from \eqref{first order of q at boundary}. Since we also know the derivatives of $q$ in tangential directions $x_k$, where $k=1,\ldots, n-1$, we have determined all first order derivatives of $q$ on the boundary. 
% 
% To determine the high
% 
% 	Using the fact that $u_0=f_0$ is known on the boundary $\p M$, it follows that $\p_{x_k}u_0$ is determined along the tangential direction near $x=x_0$, for $k=1,2,\ldots,n-1$. 
% 	Thus, $\p_{x_k} q$ is determined due to the identity \eqref{first order of q at boundary} near $x=x_0$. On the other hand, for $k=n$, we know that 
% 	$$
% 	\p_{x_n}u_0=\Lambda_q(f_0) \text{ near } x=x_0\in \p M.
% 	$$ 
% 	By using \eqref{first order of q at boundary} again, we can also determine $\p_{x_n}q$. 
	
	To determine higher order derivatives of $q$ on the boundary, we follow an argument similar to~\cite[Lemma 3.4]{lassas2016conformal}. %Given a point $x=x_0 \in \p M$, we choose boundary normal coordinates $x=(x',x_n)$ 
	On a neighborhood of $x_0$ in $M$ we may write 
	\[
	Qu_0:=(-\Delta_g +V) u_0+qu_0^m=-\p_{x_n}^2 u_0+Pu_0,
	\]
	where $P$ is a non-linear partial differential operator containing derivatives in $x'$ up to order $2$ and in $x_n$ up to order $1$. The coefficients of $P$ depend on pointwise values of $q$. By expressing
	\[
	\p_{x_n}^2=P-Q                  
	\]
	we obtain
	\begin{equation}\label{bnd_det_eq}
		\p_{x_n}^2u_0=Pu_0-Qu_0=Pu_0.
	\end{equation}
	Since we already know the quantities % know $\p_{x_n}u=\Lambda(f)$ on a neighborhood of $x_0$ in $\p M$ and formal Taylor series of $q$ at $x_0$, we also know 
	\begin{align}\label{BD for derivatives of solutions and potentials}
		u_0,\  \p_{x'}u_0,\ \p_{x'}^2u_0,\ \p_{x_n}u_0,\ \p_{x'}\p_{x_n}u_0,\ q,\ \p_{x'}q  \text{ and }\p_{x_n}q, 
	\end{align}
	it follows from \eqref{bnd_det_eq} that the second derivative $\p_{x_n}^2 u_0$ can be also determined. By using this and differentiating  \eqref{first order of q at boundary}, we may determine second order derivatives of $q$ on the boundary.
	%Rest of the proof follows by repeatedly differentiating \eqref{bnd_det_eq} and using induction.
% 	Meanwhile, by using the following identity that 
% 	\[
% 	\p_{x_n}^2\widetilde{q}=mq\p_{x_n}^2(u_0^{m-1})+2m\p_{x_n}(u_0^{m-1})\p_{x_n}q+mu_0^{m-1}\p_{x_n}^2q + \p_{x_n}^2 V,
% 	\]
% 	we can determine $\p_{x_n}^2q$ on $\p M$, since $u_0=f_0$ is known a priori.  Consequently,  we  know the formal Taylor series of $q$ on $\p M$ up to second order.
% 	
	The higher order derivatives of $q$ on the boundary can be determined by differentiating \eqref{bnd_det_eq}  and using \eqref{first order of q at boundary} in succession, and by using induction.
	
% 	We see that the quantity $\p_{x_n} Pu_0$ contains only derivatives of $u_0$ up to order $2$. Thus, second derivatives of $u$ are determined. Rest of the proof follows by induction.
% % 	The rest of the proof is to combine \eqref{BD for derivatives of solutions and potentials} and differentiate~\eqref{bnd_det_eq}, then an induction argument proves the assertion.

\medskip

{\it Determination of Taylor expansions of solutions:}
Let then $f\in C^{2,\alpha}(\p M)$ be small enough so that \eqref{eq:boundary_det_equation_app} has a unique small solution $u=u_f$. Since we have determined the formal Taylor series of $q$ on the boundary, the formal Taylor series of $u$ on the boundary is determined by differentiating \eqref{bnd_det_eq} with $u$ in place of $u_0$.
	
\end{proof}

\section{Proof of the Carleman estimate with boundary terms} \label{Section: Boundary Carleman}
	In this section, we proceed to prove Lemma~\ref{lem0}. Let $(M,g)$ be a  compact, smooth, transversally anisotropic Riemannian manifold with a smooth boundary and let $V\in L^{\infty}(M)$. There exists constants $\tau_0>0$ and $C>0$ depending only on $(M,g)$ and $\|V\|_{L^{\infty}(M)}$ such that given any $|\tau|>\tau_0$ and any $v\in C^2(M)$, the following Carleman estimate holds
\begin{multline}
\label{est_0} 
\|e^{-\tau t}(-\Delta_g+V)(e^{\tau t}v)\|_{L^2(M)}+|\tau|^{\frac{3}{2}}\|v\|_{W^{2,\infty}(\p M)} +|\tau|^{\frac{3}{2}}\|\p_\nu v\|_{W^{1,\infty}(\p M)}\\
+|\tau|^{\frac{3}{2}}\|\p^2_\nu v\|_{L^\infty(\p M)} \geq C|\tau|\,\|v\|_{L^2(M)},
\end{multline}
% \HOX{The estimate (0.1) is by no means sharp, but it is sufficient for our purposes.}
% for some $C>0$ that is independent of $\tau$.

\begin{proof}[Proof of Lemma~\ref{lem0}]
We may assume without loss of generality that $v$ is real-valued and also that $\tau>0$. The proof for the case $\tau<0$ follows analogously. Throughout this proof, we use the notation $C$ to stand for a generic positive constant that is independent of the parameter $\tau$. We also write $\hat{v}$ to stand for a $C^2$-extension of the function $v$ into a slightly larger manifold $\hat{M}\Subset \R\times M_0$ with smooth boundary, such that $v\in C^{2}_c(\hat{M})$ and that there holds
\begin{equation}
\label{extension_v_bound}
\|\hat{v}\|_{W^{2,\infty}(\hat{M}\setminus M)}\leq C(\|\p_\nu^2 v\|_{L^{\infty}(\p M)}+\|\p_\nu v\|_{W^{1,\infty}(\p M)}+\|v\|_{W^{2,\infty}(\p M)}),
\end{equation}
for some constant $C>0$, only depending on $(\hat M,g)$. We remark that this extension of $v$ can always be achieved using a Taylor series approximation together with the boundary normal coordinates near $\p M$. We define
\begin{equation}
\label{P_tau}
 P_\tau v=e^{-\tau t}\Delta_g(e^{\tau t}v),
 \end{equation}
and note that 
$$P_\tau v= \p^2_t v + \Delta_{g_0}v+2\tau\p_tv + \tau^2v.$$
We claim that
\begin{multline}
\label{est_1}
\left|\int_M P_\tau v\, \p_t v\,dV_g\right| + C\tau^2\|v\|_{W^{2,\infty}(\p M)}^2+C\tau^2\|\p_\nu v\|_{W^{1,\infty}(\p M)}^2\\
+C\tau^2\|\p^2_\nu v\|_{L^{\infty}(\p M)}^2
\geq 2\tau \|\p_t v\|_{L^2(M)}^2.
\end{multline}
To show \eqref{est_1} we begin by writing
\begin{align*}
\begin{split}
	\int_M P_\tau v\, \p_t v\,dV_g = &2\tau \int_M |\p_t v|^2 \,dV_g \\
	& +\underbrace{\int_M \p_t^2v\,\p_t v\, dV_g}_{\text{I}}+\underbrace{\int_M \Delta_{g_0}v\,\p_t v\, dV_g}_{\text{II}}
	+\underbrace{\int_M\tau^2v\,\p_t v\, dV_g}_{\text{III}}.
\end{split}
\end{align*}
Note that $M \Subset \R\times M_0$ and $dV_g= dt\,dV_{g_0}$. We can use integration by parts to bound each of the terms I--III as follows.
For I, we first note that
$$\int_{\hat{M}} \p_t^2\hat v\,\p_t \hat v\, dV_g=\frac{1}{2}\int_{\hat M} \p_t(|\p_t \hat v|^2)\,dV_g
=0.$$ 
Together with the estimate \eqref{extension_v_bound}, we obtain
\begin{multline*}
|\text{I}|= \left|\int_{\hat M\setminus M}\p_t^2\hat v\,\p_t \hat v\, dV_g\right|
\leq C\left(\|\p_\nu^2 v\|^2_{L^{\infty}(\p M)}+\|\p_\nu v\|^2_{W^{1,\infty}(\p M)}+\|v\|^2_{W^{2,\infty}(\p M)}\right).\end{multline*}
For II, since $\left[\p_t, \Delta_g\right]=0$ on $(\hat M,g)$, we may apply integration by parts again to deduce that
$$\int_{\hat M} \Delta_{g_0}\hat v\,\p_t \hat v\, dV_g=0.$$
Thus, using \eqref{extension_v_bound}, we can show analogously to term I that
$$ |\text{II}| \leq C\left(\|\p_\nu^2 v\|^2_{L^{\infty}(\p M)}+\|\p_\nu v\|^2_{W^{1,\infty}(\p M)}+\|v\|^2_{W^{2,\infty}(\p M)}\right).$$
Finally for the term III we first note that
$$\tau^2\int_M \hat v\,\p_t \hat v\, dV_g=\frac{\tau^2}{2}\int_M \p_t(\hat{v}^2)\,dV_g=0.$$
Thus, using \eqref{extension_v_bound}, we have
$$|\text{III}|\leq 
C\tau^2\left(\|\p_\nu^2 v\|^2_{L^{\infty}(\p M)}+\|\p_\nu v\|^2_{W^{1,\infty}(\p M)}+\|v\|^2_{W^{2,\infty}(\p M)}\right).$$
Combining the previous three bounds yields the claimed estimate \eqref{est_1}. Using \eqref{est_1} and applying the Cauchy-Schwarz inequality 
$$\left|\int_M P_\tau v\, \p_t v\,dV_g\right| \leq \frac{1}{4\tau}\|P_\tau v\|_{L^2(M)}^2 + \tau\|\p_t v\|_{L^2(M)}^2,$$
we deduce that
\begin{multline}
\label{est_2}
\|P_\tau v\|_{L^2(M)}^2+ C\tau^3\|v\|_{W^{2,\infty}(\p M)}^2+C\tau^3\|\p_\nu v\|_{W^{1,\infty}(\p M)}^2\\
+C\tau^3\|\p^2_\nu v\|_{L^{\infty}(\p M)}^2
\geq \tau^2 \|\p_t v\|_{L^2(M)}^2,
\end{multline}
We recall that by the standard Poincar\'{e} inequality on $\hat{M}$, there exists $C>0$ such that
$$ \|\p_t w\|_{L^2(\hat{M})} \geq  C\|w\|_{L^2(\hat{M})}\qquad \forall w \in H^1_0(\hat{M}).$$
Using the latter bound together with a bound analogous to \eqref{extension_v_bound} for extending $C^1(M)$ functions into $C^{1}_c(\hat{M})$, we can derive the following Poincar\'{e} type inequality
\begin{equation}
\label{poincare_bd}
\|\p_t v\|_{L^2(M)} \geq  C_1\|v\|_{L^2(M)} - C_2\|v\|_{W^{1,\infty}(\p M)}-C_3\|\p_\nu v\|_{L^{\infty}(\p M)},
\end{equation}
for all $v\in C^1(M)$, where the positive constants $C_1$, $C_2$ and $C_3$ only depend on $(M,g)$. 

Via the bounds \eqref{est_2}--\eqref{poincare_bd}, we deduce that
\begin{multline}
\label{est_3}
\|(P_\tau-V) v\|_{L^2(M)}^2+ C\tau^3\|v\|_{W^{2,\infty}(\p M)}^2+C\tau^3\|\p_\nu v\|_{W^{1,\infty}(\p M)}^2\\
+C\tau^3\|\p^2_\nu v\|_{L^{\infty}(\p M)}^2
\geq \tau^2 \|v\|_{L^2(M)}^2.
\end{multline} 
This proves the assertion.
\end{proof}

\section{Computations of $\mathbf{D}_{ik}$}\label{Appendix D_ikl}

In the end of this paper, we compute the values $\mathbf{D}_{ik}$, for different sub-indices $i,k\in \{1,2,3,4,5\}$. Recalling that
\begin{align*}
	&\zeta_1=\xi_1, \qquad \zeta_2=\xi_2, \qquad  \\
	&\zeta_3=
	\left(1+\sqrt{\frac{2}{2-\delta}}\right)\xi_3 ,  \quad \zeta_4=\left(1+\sqrt{\frac{2}{2-\delta}}\right)\xi_4, \\
	&\zeta_5=\sqrt{\frac{2}{2-\delta}}(\xi_1+\xi_2),
\end{align*}
and
\begin{align*}
	& \ol \zeta_1=\abs{\zeta_1}e_1+ \mathbf{i}\zeta_1, \quad \ol \zeta_2=\abs{\zeta_2}e_1+\mathbf{i}\zeta_2, \\
	&\ol \zeta_3=\abs{\zeta_3}e_1+\mathbf{i}\zeta_3, \quad \ol \zeta_4=-\abs{\zeta_4}e_1+\mathbf{i}\zeta_4,\\
	& \ol \zeta_5=-\abs{\zeta_5}e_1+\mathbf{i}\zeta_5,
\end{align*}
where 
\begin{align*}
	&\abs{\xi_1}=\abs{\xi_2}=1, \quad \langle \xi_1 , \xi_2 \rangle =1-\delta,\\
	&\xi_3=-\frac{1}{1+\delta}(\xi_1+\delta \xi_2), \quad \xi_4=-\frac{1}{1+\delta}(\delta\xi_1 + \xi_2).
\end{align*}
Via straightforward computations, we have 
\begin{align*}
 &\langle \xi_1 , \xi_2\rangle = 1- \delta, \quad  \langle \xi_1,\xi_3 \rangle=-\frac{1+\delta-\delta^2}{1+\delta}, \quad  \langle \xi_1, \xi_4\rangle=-\frac{1}{1+\delta}, \\
 &\langle \xi_2 , \xi_3\rangle = -\frac{1}{1+\delta}, \quad  \langle \xi_2,\xi_4 \rangle=-\frac{1+\delta-\delta^2}{1+\delta},  \text{ and }  \langle \xi_3,\xi_4 \rangle =\frac{1+\delta+\delta^2-\delta^3}{(1+\delta)^2}.
\end{align*}

By 
$$
\mathbf{D}_{ij}=\langle \ol \zeta_i+\ol \zeta_j,  \ol \zeta_i+\ol \zeta_j \rangle,
$$
for different $i,k\in \{1,2,3,4,5\}$, direct computations yield that 
\begin{align}\label{D_12}
	\begin{split}
		\mathbf{D}_{12}= &\left( \abs{\zeta_1}e_1+ \mathbf{i}\zeta_1+\abs{\zeta_2}e_1+\mathbf{i}\zeta_2\right) \cdot \left( \abs{\zeta_1}e_1+ \mathbf{i}\zeta_1+\abs{\zeta_2}e_1+\mathbf{i}\zeta_2\right) \\
		=& 2 \left( |\zeta_1||\zeta_2|-\langle \zeta_1, \zeta_2 \rangle \right) \\
		=& 2 \left( |\xi_1||\xi_2|-\langle \xi_1, \xi_2 \rangle \right) = 2\delta,
	\end{split}
\end{align}
\begin{align}\label{D_13}
\begin{split}
		\mathbf{D}_{13}= & \left(\abs{\zeta_1}e_1+ \mathbf{i}\zeta_1 +\abs{\zeta_3}e_1+\mathbf{i}\zeta_3\right) \cdot  \left(\abs{\zeta_1}e_1+ \mathbf{i}\zeta_1 +\abs{\zeta_3}e_1+\mathbf{i}\zeta_3\right) \\
	=&2 \left(\abs{\zeta_1}\abs{\zeta_3}-\langle \zeta_1, \zeta_3 \rangle \right) \\
	=&2 	\left(1+\sqrt{\frac{2}{2-\delta}}\right) \left( \abs{\xi_1} \abs{\xi_3}   -\langle	\xi_1, \xi_3 \rangle   \right) \\
	=&2 	\left(1+\sqrt{\frac{2}{2-\delta}}\right) \left( 2+2\delta+\mathcal{O}(\delta^2)\right),
\end{split}
\end{align}
\begin{align}\label{D_14}
	\begin{split}
		\mathbf{D}_{14}=& \left(\abs{\zeta_1}e_1+ \mathbf{i}\zeta_1 -\abs{\zeta_4}e_1+\mathbf{i}\zeta_4\right) \cdot  \left(\abs{\zeta_1}e_1+ \mathbf{i}\zeta_1 -\abs{\zeta_4}e_1+\mathbf{i}\zeta_4\right) \\
		=&-2\left(1+\sqrt{\frac{2}{2-\delta}}\right) \left(\abs{\xi_1}\abs{\xi_4}+\langle \xi_1 ,\xi_4 \rangle \right) \\
		=&-2\left(1+\sqrt{\frac{2}{2-\delta}}\right) \frac{\delta +\mathcal{O}(\delta^2)}{1+\delta},
	\end{split}
\end{align}

\begin{align}\label{D_15}
	\begin{split}
	      \mathbf{D}_{15}=& \left( \abs{\zeta_1}e_1+ \mathbf{i}\zeta_1-\abs{\zeta_5}e_1+\mathbf{i}\zeta_5 \right) \cdot  \left( \abs{\zeta_1}e_1+ \mathbf{i}\zeta_1-\abs{\zeta_5}e_1+\mathbf{i}\zeta_5 \right) \\
	      =& -2 \left( \abs{\zeta_1}\abs{\zeta_5} + \langle \zeta_1, \zeta_5 \rangle \right) \\
	      = &-2 \sqrt{\frac{2}{2-\delta}} \left( \abs{\xi_1}\abs{\xi_1 +\xi_2} +\langle \xi_1,\xi_1+\xi_2\rangle \right) \\
	      =&-8+\frac{\delta}{2}+\mathcal{O}(\delta^2),
	\end{split}
\end{align}

\begin{align}\label{D_23}
 \begin{split}
 	\mathbf{D}_{23}=&\left(\abs{\zeta_2}e_1+\mathbf{i}\zeta_2+\abs{\zeta_3}e_1+\mathbf{i}\zeta_3 \right)\cdot \left( \abs{\zeta_2}e_1+\mathbf{i}\zeta_2+\abs{\zeta_3}e_1+\mathbf{i}\zeta_3 \right) \\
 	=& 	2	\left(1+\sqrt{\frac{2}{2-\delta}}\right)\left( \abs{\xi_2}\abs{\xi_3} -\langle \xi_2 ,\xi_3\rangle \right)  \\
 	=&2 	\left(1+\sqrt{\frac{2}{2-\delta}}\right) \left( \frac{2+\delta+\mathcal{O}(\delta^2)}{1+\delta} \right).
 \end{split}
\end{align}

In order to compute $\mathbf{D}_{24}$ more carefully, let us recall the Taylor expansion of $\sqrt{1+\delta}=1+\frac{\delta}{2}-\frac{\delta^2}{8}+\mathcal{O}(\delta^3)$, then we have 
\begin{align}\label{D_24}
	\begin{split}
		\mathbf{D}_{24} = & \left(\abs{\zeta_2}e_1+\mathbf{i}\zeta_2 -\abs{\zeta_4}e_1+\mathbf{i}\zeta_4\right) \cdot  \left(\abs{\zeta_2}e_1+\mathbf{i}\zeta_2 -\abs{\zeta_4}e_1+\mathbf{i}\zeta_4\right) \\
		=& -2 \left( \abs{\zeta_2}\abs{\zeta_4} +\langle \zeta_2, \zeta_4 \rangle \right)  \\
		=& -2\left(1+\sqrt{\frac{2}{2-\delta}}\right) \left( \abs{\xi_2}\abs{\xi_4} +\langle \xi_2, \xi_4 \rangle \right)  \\
		=& -2\left(1+\sqrt{\frac{2}{2-\delta}}\right) \left( \sqrt{1+2\delta-\delta^2}-(1+\delta-\delta^2) \right) \\
		=& -2\left(1+\sqrt{\frac{2}{2-\delta}}\right)  \left(1 +\delta -\delta^2 +\mathcal{O}(\delta^3)-(1+\delta-\delta^2) \right)\\
		=&-2 \left(1+\sqrt{\frac{2}{2-\delta}}\right)  \frac{\mathcal{O}(\delta^3)}{1+\delta},
	\end{split}
\end{align}

\begin{align}\label{D_25}
	\begin{split}
		\mathbf{D}_{25} =& \left(\abs{\zeta_2}e_1+\mathbf{i}\zeta_2-\abs{\zeta_5}e_1+\mathbf{i}\zeta_5 \right)\cdot \left(\abs{\zeta_2}e_1+\mathbf{i}\zeta_2-\abs{\zeta_5}e_1+\mathbf{i}\zeta_5 \right) \\
		=& -2 \left( \abs{\zeta_2}\abs{\zeta_5}  +\langle \zeta_2,\zeta_5 \rangle \right) \\
		=&-2\sqrt{\frac{2}{2-\delta}} \left(\sqrt{4-2\delta}+2-\delta \right) \\
		=&-8+\delta +\mathcal{O}(\delta^2),
	\end{split}
\end{align}

\begin{align}\label{D_34}
	\begin{split}
		\mathbf{D}_{34} =& \left( \abs{\zeta_3}e_1+\mathbf{i}\zeta_3-\abs{\zeta_4}e_1+\mathbf{i}\zeta_4 \right) \cdot \left( \abs{\zeta_3}e_1+\mathbf{i}\zeta_3-\abs{\zeta_4}e_1+\mathbf{i}\zeta_4 \right) \\
		=& -2 \left( \abs{\zeta_3}\abs{\zeta_4} + \langle \zeta_3,\zeta_4 \rangle \right) \\
		=& -2\left(1+\sqrt{\frac{2}{2-\delta}}\right)^2\left( \abs{\xi_3}\abs{\xi_4} + \langle \xi_3,\xi_4 \rangle \right) \\
		=& -2\left(1+\sqrt{\frac{2}{2-\delta}}\right)^2 \frac{1}{(1+\delta)^2} \left(2+3\delta  -\delta^3 \right) ,
	\end{split}
\end{align}

\begin{align}\label{D_35}
	\begin{split}
		\mathbf{D}_{35}=& \left( \abs{\zeta_3}e_1+\mathbf{i}\zeta_3-\abs{\zeta_5}e_1+\mathbf{i}\zeta_5 \right) \cdot  \left( \abs{\zeta_3}e_1+\mathbf{i}\zeta_3-\abs{\zeta_5}e_1+\mathbf{i}\zeta_5 \right) \\
		=& -2 \left( \abs{\zeta_3}\abs{\zeta_5}+\langle \zeta_3 , \zeta_5\rangle \right) \\
		= &-\frac{2}{1+\delta}	\left(1+\sqrt{\frac{2}{2-\delta}}\right)\sqrt{\frac{2}{2-\delta}} \left( \sqrt{4-2\delta}(1+\delta+\mathcal{O}(\delta^2)) - 2-\delta+\delta^2\right) \\
		=&-\frac{2}{1+\delta}	\left(1+\sqrt{\frac{2}{2-\delta}}\right)\sqrt{\frac{2}{2-\delta}} \\
		& \qquad \times \left((2-\frac{\delta}{2}+\mathcal{O}(\delta^2))(1+\delta+\mathcal{O}(\delta^2)) - 2-\delta+\delta^2\right)  \\
		=&-\frac{1}{1+\delta}	\left(1+\sqrt{\frac{2}{2-\delta}}\right)\sqrt{\frac{2}{2-\delta}} \left( \delta+\mathcal{O}(\delta^2) \right) 
	\end{split}
\end{align}
and similarly,
\begin{align}\label{D_45}
	\begin{split}
		\mathbf{D}_{45}=& \left( -\abs{\zeta_4}e_1+\mathbf{i}\zeta_4-\abs{\zeta_5}e_1+\mathbf{i}\zeta_5 \right) \cdot  \left( -\abs{\zeta_4}e_1+\mathbf{i}\zeta_4-\abs{\zeta_5}e_1+\mathbf{i}\zeta_5 \right) \\
		=&2 \left( \abs{\zeta_4}\abs{\zeta_5}-\langle \zeta_4, \zeta_5 \rangle \right)  \\
		= & 2 \left(1+\sqrt{\frac{2}{2-\delta}}\right)\sqrt{\frac{2}{2-\delta}} \left( \abs{\xi_4}\abs{\xi_1 +\xi_2}-\langle \xi_4, \xi_1+\xi_2 \rangle \right)  \\
		=&\frac{2}{1+\delta} \left(1+\sqrt{\frac{2}{2-\delta}}\right)\sqrt{\frac{2}{2-\delta}} \left( \sqrt{4-2\delta}(1+\delta+\mathcal{O}(\delta^2))+2+\delta-\delta^2\right)\\
		=&\frac{2}{1+\delta} \left(1+\sqrt{\frac{2}{2-\delta}}\right)\sqrt{\frac{2}{2-\delta}}  \left(4+\frac{5}{2}\delta + \mathcal{O}(\delta^2)\right).
	\end{split}
\end{align}

\begin{proof}[Proof of Lemma \ref{lem:Edelta}]
	With \eqref{D_12}--\eqref{D_45} at hand, let us split the analysis into two cases.
	
    \noindent\textbf{(1)} By using \eqref{D_23}, \eqref{D_24} and \eqref{D_34}, we have that $\frac{1}{\mathbf{D}_{23}+\mathbf{D}_{24}+\mathbf{D}_{34}}$ is a bounded as $\delta\to 0$.
		Similarly, \eqref{D_13}, \eqref{D_14} and \eqref{D_34} imply that $\frac{1}{\mathbf{D}_{13}+\mathbf{D}_{14}+\mathbf{D}_{34}}$ is also bounded as $\delta\to 0$. Similarly $\frac{1}{\mathbf{D}_{12}+\mathbf{D}_{13}+\mathbf{D}_{23}}$ is bounded as $\delta\to 0$.  On the other hand, by \eqref{D_12}, \eqref{D_14} and \eqref{D_24}, we observe that  $\frac{1}{\mathbf{D}_{12}+\mathbf{D}_{14}+\mathbf{D}_{24}}=\mathcal{O}(\delta^{-1})$. 
		Meanwhile, $\mathbf{D}_{15}^{-1}$, $\mathbf{D}_{25}^{-1}$ and $\mathbf{D}_{45}^{-1}$ are bounded as $\delta \to 0$, but $\mathbf{D}_{35}^{-1}=\mathcal{O}(\delta^{-1})$. 
		
		\medskip

	\noindent\textbf{(2)}  Similarly, $\frac{1}{\mathbf{D}_{12}}\frac{1}{\mathbf{D}_{34}}=\mathcal{O}(\delta^{-1})$,   $\frac{1}{\mathbf{D}_{13}}\frac{1}{\mathbf{D}_{24}}=\mathcal{O}(\delta^{-3})$ and  $\frac{1}{\mathbf{D}_{14}}\frac{1}{\mathbf{D}_{23}}=\mathcal{O}(\delta^{-1})$.
	
	\medskip
	
	Therefore, combining the above, we conclude that 
	\begin{align*}
		\begin{split}
			\mathbf{E}_\delta =& \left| \frac{1}{\mathbf{D}_{15}}\left(\frac{1}{\mathbf{D}_{23}+\mathbf{D}_{24}+\mathbf{D}_{34}}\right) + \frac{1}{\mathbf{D}_{25}}\left(\frac{1}{\mathbf{D}_{13}+\mathbf{D}_{14}+\mathbf{D}_{34}}\right)   \right. \\		 
			&\qquad  + \frac{1}{\mathbf{D}_{35}}\left(\frac{1}{\mathbf{D}_{12}+\mathbf{D}_{14}+\mathbf{D}_{24}}\right) + \frac{1}{\mathbf{D}_{45}}\left(\frac{1}{\mathbf{D}_{12}+\mathbf{D}_{13}+\mathbf{D}_{23}}\right) \\
			&\qquad   \left.  + \frac{1}{\mathbf{D}_{12}}\frac{1}{\mathbf{D}_{34}}+  \frac{1}{\mathbf{D}_{13}}\frac{1}{\mathbf{D}_{24}} + \frac{1}{\mathbf{D}_{14}}\frac{1}{\mathbf{D}_{23}}\right| \\
			\geq &\frac{C_0}{\delta^3} -\frac{C_1}{\delta^2}-C_2>0,
		\end{split}
	\end{align*}
	for all sufficiently small $\delta>0$, where $C_0$, $C_1$ and $C_2$ are some positive constants independent of $\delta$. Hence, the coefficient $\mathbf{E}_\delta=\mathcal{O}(\delta^{-3})\neq 0$ 
	%in the formula  \eqref{w_eq_4th_choosing_vectors_formal} will not vanish 
	for all sufficiently small $\delta>0$.
\end{proof}

\vskip0.5cm

\noindent\textbf{Acknowledgment.} A.F gratefully acknowledges support of the Fields institute for research in mathematical sciences. T.L. was supported by the Academy of Finland (Centre of Excellence in Inverse Modeling and Imaging, grant numbers 284715 and 309963). The work of Y.-H. Lin is partially supported by the Ministry of Science and Technology Taiwan, under the Columbus Program: MOST-110-2636-M-009-007.

\bibliographystyle{alpha}
\bibliography{ref}

\newcommand{\etalchar}[1]{$^{#1}$}
\begin{thebibliography}{DSFKSU09}

\bibitem[CF20]{CF20}
C{\u{a}}t{\u{a}}lin~I C{\^a}rstea and Ali Feizmohammadi.
\newblock A density property for tensor products of gradients of harmonic
  functions and applications.
\newblock {\em arXiv preprint arXiv:2009.11217}, 2020.

\bibitem[CF21]{MR4227095}
C\u{a}t\u{a}lin~I. C\^{a}rstea and Ali Feizmohammadi.
\newblock An inverse boundary value problem for certain anisotropic quasilinear
  elliptic equations.
\newblock {\em J. Differential Equations}, 284:318--349, 2021.

\bibitem[CFK{\etalchar{+}}21]{MR4300916}
C\u{a}t\u{a}lin~I. C\^{a}rstea, Ali Feizmohammadi, Yavar Kian, Katya Krupchyk,
  and Gunther Uhlmann.
\newblock The {C}alder\'{o}n inverse problem for isotropic quasilinear
  conductivities.
\newblock {\em Adv. Math.}, 391:Paper No. 107956, 31, 2021.

\bibitem[CNV19]{MR3964222}
C\u{a}t\u{a}lin~I. C\^{a}rstea, Gen Nakamura, and Manmohan Vashisth.
\newblock Reconstruction for the coefficients of a quasilinear elliptic partial
  differential equation.
\newblock {\em Appl. Math. Lett.}, 98:121--127, 2019.

\bibitem[DSFKSU09]{DosSantosFerreira2009}
David Dos Santos~Ferreira, Carlos~E. Kenig, Mikko Salo, and Gunther Uhlmann.
\newblock Limiting carleman weights and anisotropic inverse problems.
\newblock {\em Inventiones mathematicae}, 178(1):119--171, Oct 2009.

\bibitem[FKLS16]{ferreira2013calderon}
David Dos~Santos Ferreira, Yaroslav Kurylev, Matti Lassas, and Mikko Salo.
\newblock The {C}alder{\'o}n problem in transversally anisotropic geometries.
\newblock {\em J. Eur. Math. Soc. (JEMS)}, 18:2579--2626, 2016.

\bibitem[FKSU09]{ferreira2009limiting}
David Dos~Santos Ferreira, Carlos Kenig, Mikko Salo, and Gunther Uhlmann.
\newblock Limiting {C}arleman weights and anisotropic inverse problems.
\newblock {\em Inventiones mathematicae}, 178(1):119--171, 2009.

\bibitem[FLO21]{MR4336252}
Ali Feizmohammadi, Matti Lassas, and Lauri Oksanen.
\newblock Inverse problems for nonlinear hyperbolic equations with disjoint
  sources and receivers.
\newblock {\em Forum Math. Pi}, 9:Paper No. e10, 2021.

\bibitem[FO20]{FO19}
Ali Feizmohammadi and Lauri Oksanen.
\newblock An inverse problem for a semi-linear elliptic equation in
  {R}iemannian geometries.
\newblock {\em Journal of Differential Equations}, 269(6):4683--4719, 2020.

\bibitem[HS02]{MR1944036}
David Hervas and Ziqi Sun.
\newblock An inverse boundary value problem for quasilinear elliptic equations.
\newblock {\em Comm. Partial Differential Equations}, 27(11-12):2449--2490,
  2002.

\bibitem[HUZ20]{hintz2020inverse}
Peter Hintz, Gunther Uhlmann, and Jian Zhai.
\newblock An inverse boundary value problem for a semilinear wave equation on
  lorentzian manifolds.
\newblock {\em arXiv preprint arXiv:2005.10447}, 2020.

\bibitem[IN95]{victorN}
Victor Isakov and A~Nachman.
\newblock Global uniqueness for a two-dimensional elliptic inverse problem.
\newblock {\em Trans.of AMS}, 347:3375--3391, 1995.

\bibitem[IS94]{isakov1994global}
Victor Isakov and John Sylvester.
\newblock Global uniqueness for a semilinear elliptic inverse problem.
\newblock {\em Communications on Pure and Applied Mathematics},
  47(10):1403--1410, 1994.

\bibitem[Isa93]{isakov1993uniqueness}
Victor Isakov.
\newblock On uniqueness in inverse problems for semilinear parabolic equations.
\newblock {\em Archive for Rational Mechanics and Analysis}, 124(1):1--12,
  1993.

\bibitem[KLU18]{kurylev2018inverse}
Yaroslav Kurylev, Matti Lassas, and Gunther Uhlmann.
\newblock Inverse problems for {L}orentzian manifolds and non-linear hyperbolic
  equations.
\newblock {\em Inventiones mathematicae}, 212(3):781--857, 2018.

\bibitem[KU20a]{MR3168271}
Yavar Kian and Gunther Uhlmann.
\newblock Recovery of nonlinear terms for reaction diffusion equations from
  boundary measurements.
\newblock {\em arXiv preprint arXiv:2011.06039}, 2020.

\bibitem[KU20b]{KrUh20}
Katya Krupchyk and Gunther Uhlmann.
\newblock {Inverse problems for nonlinear magnetic Schrödinger equations on
  conformally transversally anisotropic manifolds}.
\newblock {\em arXiv e-prints arXiv:2009.05089}, 2020.

\bibitem[KU20c]{MR4052205}
Katya Krupchyk and Gunther Uhlmann.
\newblock A remark on partial data inverse problems for semilinear elliptic
  equations.
\newblock {\em Proc. Amer. Math. Soc.}, 148(2):681--685, 2020.

\bibitem[Lin21]{lin2020monotonicity}
Yi-Hsuan Lin.
\newblock Monotonicity-based inversion of fractional semilinear elliptic
  equations with power type nonlinearities.
\newblock {\em Calculus of Variations and Partial Differential Equations},
  accepted for publication, 2021.

\bibitem[LL22]{lai2022inverse}
Ru-Yu Lai and Yi-Hsuan Lin.
\newblock Inverse problems for fractional semilinear elliptic equations.
\newblock {\em Nonlinear Analysis}, 216:112699, 2022.

\bibitem[LLL21]{lin2021determining}
Yi-Hsuan Lin, Hongyu Liu, and Xu~Liu.
\newblock Determining a nonlinear hyperbolic system with unknown sources and
  nonlinearity.
\newblock {\em arXiv preprint arXiv:2107.10219}, 2021.

\bibitem[LLLS21a]{LLLS2019inverse}
Matti Lassas, Tony Liimatainen, Yi-Hsuan Lin, and Mikko Salo.
\newblock Inverse problems for elliptic equations with power type
  nonlinearities.
\newblock {\em Journal de math{\'e}matiques pures et appliqu{\'e}es},
  145:44--82, 2021.

\bibitem[LLLS21b]{LLLS2021b}
Matti Lassas, Tony Liimatainen, Yi-Hsuan Lin, and Mikko Salo.
\newblock Partial data inverse problems and simultaneous recovery of boundary
  and coefficients for semilinear elliptic equations.
\newblock {\em Revista Matemática Iberoamericana}, 37:1553--1580, 2021.

\bibitem[LLLZ21]{lin2021simultaneous}
Yi-Hsuan Lin, Hongyu Liu, Xu~Liu, and Shen Zhang.
\newblock Simultaneous recoveries for semilinear parabolic systems.
\newblock {\em arXiv preprint arXiv:2111.05242}, 2021.

\bibitem[LLPMT21]{lassas2021stability}
Matti Lassas, Tony Liimatainen, Leyter Potenciano-Machado, and Teemu Tyni.
\newblock Stability estimates for inverse problems for semi-linear wave
  equations on lorentzian manifolds.
\newblock {\em arXiv preprint arXiv:2106.12257}, 2021.

\bibitem[LLS16]{lassas2016conformal}
Matti Lassas, Tony Liimatainen, and Mikko Salo.
\newblock {The Calder{\'o}n problem for the conformal Laplacian}.
\newblock {\em arXiv e-prints 1612.07939}, 2016.

\bibitem[LLS20]{LLS18}
Matti Lassas, Tony Liimatainen, and Mikko Salo.
\newblock The {P}oisson embedding approach to the {C}alder\'{o}n problem.
\newblock {\em Math. Ann.}, 377(1-2):19--67, 2020.

\bibitem[LLST22]{MR4332042}
Tony Liimatainen, Yi-Hsuan Lin, Mikko Salo, and Teemu Tyni.
\newblock Inverse problems for elliptic equations with fractional power type
  nonlinearities.
\newblock {\em J. Differential Equations}, 306:189--219, 2022.

\bibitem[LUW17]{lassas2017determination}
Matti Lassas, Gunther Uhlmann, and Yiran Wang.
\newblock Determination of vacuum space-times from the {E}instein-{M}axwell
  equations.
\newblock {\em arXiv preprint arXiv:1703.10704}, 2017.

\bibitem[LUW18]{lassas2018inverse}
Matti Lassas, Gunther Uhlmann, and Yiran Wang.
\newblock Inverse problems for semilinear wave equations on {L}orentzian
  manifolds.
\newblock {\em Communications in Mathematical Physics}, 360:555--609, 2018.

\bibitem[MnU20]{MR4138229}
Claudio Mu\~{n}oz and Gunther Uhlmann.
\newblock The {C}alder\'{o}n problem for quasilinear elliptic equations.
\newblock {\em Ann. Inst. H. Poincar\'{e} Anal. Non Lin\'{e}aire},
  37(5):1143--1166, 2020.

\bibitem[Sha21]{MR4197837}
Ravi Shankar.
\newblock Recovering a quasilinear conductivity from boundary measurements.
\newblock {\em Inverse Problems}, 37(1):Paper No. 015014, 24, 2021.

\bibitem[SU97]{sun1997inverse}
Ziqi Sun and Gunther Uhlmann.
\newblock Inverse problems in quasilinear anisotropic media.
\newblock {\em American Journal of Mathematics}, 119(4):771--797, 1997.

\bibitem[Sun96]{sun1996}
Ziqi Sun.
\newblock On a quasilinear inverse boundary value problem.
\newblock {\em Math. Z.}, 221(2):293--305, 1996.

\bibitem[Sun04]{sun2004inverse}
Ziqi Sun.
\newblock Inverse boundary value problems for a class of semilinear elliptic
  equations.
\newblock {\em Advances in Applied Mathematics}, 32(4):791--800, 2004.

\bibitem[Sun10]{sun2010inverse}
Ziqi Sun.
\newblock An inverse boundary-value problem for semilinear elliptic equations.
\newblock {\em Electronic Journal of Differential Equations (EJDE)[electronic
  only]}, 37:1--5, 2010.

\bibitem[UZ21a]{MR4182329}
Gunther Uhlmann and Jian Zhai.
\newblock Inverse problems for nonlinear hyperbolic equations.
\newblock {\em Discrete Contin. Dyn. Syst.}, 41(1):455--469, 2021.

\bibitem[UZ21b]{MR4299822}
Gunther Uhlmann and Jian Zhai.
\newblock On an inverse boundary value problem for a nonlinear elastic wave
  equation.
\newblock {\em J. Math. Pures Appl. (9)}, 153:114--136, 2021.

\bibitem[WZ19]{wang2016quadratic}
Yiran Wang and Ting Zhou.
\newblock Inverse problems for quadratic derivative nonlinear wave equations.
\newblock {\em Communications in Partial Differential Equations},
  44(11):1140--1158, 2019.

\end{thebibliography}

\end{document}